\documentclass[11pt,letterpaper]{amsart}
\usepackage{amssymb,amsmath,amscd,amsbsy,amsthm,amsfonts,hyperref,accents,fontenc,bbm}

\input xy
\xyoption{all}

\newcommand{\aaa}{{\Bbb A}}
\newcommand{\zz}{{\Bbb Z}}

\newcommand{\nn}{{\Bbb N}}

\newcommand{\qq}{{\Bbb Q}}
\newcommand{\rr}{{\Bbb R}}
\newcommand{\pp}{{\Bbb P}}
\newcommand{\ff}{{\Bbb F}}
\newcommand{\row}{\rightarrow}
\newcommand{\llow}{\longleftarrow}

\newcommand{\lrow}{\longrightarrow}
\renewcommand{\leq}{\leqslant}
\renewcommand{\geq}{\geqslant}
\newcommand{\nichego}[1]{}
\newcommand{\ov}[1]{\overline{#1}}

\newcommand{\wt}[1]{\widetilde{#1}}
\newcommand{\op}[1]{\operatorname{#1}}
\newcommand{\Hom}{\op{Hom}}
\newcommand{\ddim}{\op{dim}}

\newcommand{\laz}{{\Bbb L}}
\newcommand{\CH}{\op{CH}}
\newcommand{\Ch}{\op{Ch}}
\newcommand{\ca}{{\mathcal A}}
\newcommand{\hi}{{\mathcal X}}
\newcommand{\hit}{\widetilde{{\mathcal X}}}
\newcommand{\End}{{\Bbb E}nd}
\newcommand{\HHom}{{\Bbb H}om}
\newcommand{\ffi}{\varphi}

\newcommand{\la}{\langle}
\newcommand{\ra}{\rangle}
\newcommand{\Spc}{\op{Spc}}
\newcommand{\smk}{{\mathbf Sm}_k}
\newcommand{\SH}{\op{SH}}
\newcommand{\MGL}{\op{MGL}}
\newcommand{\DM}{\op{DM}}

\newcommand{\Qed}{\hfill$\square$\smallskip}
\newcommand{\Red}{\hfill$\triangle$\smallskip}

\renewenvironment{proof}{\noindent{\it Proof}:}{\vskip 5mm}

\theoremstyle{plain}
\newtheorem{proposition}{Proposition}[section]{\bf}{\it}
\newtheorem{theorem}[proposition]{Theorem}{\bf}{\it}
\newtheorem{lemma}[proposition]{Lemma}{\bf}{\it}
{\bf}{\it}
\newtheorem{definition}[proposition]{Definition}{\bf}{\rm}
{\bf}{\it}
{\bf}{\it}
\newtheorem{example}[proposition]{Example}{\bf}{\it}
\newtheorem{remark}[proposition]{Remark}{\bf}{\rm}
{\bf}{\rm}
\newtheorem{lem}{Lemma}[proposition]{\bf}{\it}
{\bf}{\it}
{\bf}{\it}

{\bf}{\it}
\newtheorem{corollary}[proposition]{Corollary}{\bf}{\it}
{\bf}{\it}

\begin{document}
	
	\title{On the Balmer spectrum of the Morel-Voevodsky category}
	\author{Peng Du \ {\small and} \ Alexander Vishik}
	\date{}
	\maketitle
	
	\begin{abstract}
		We introduce the {\it Morava-isotropic stable homotopy category}
		and, more generally, the {\it stable homotopy category of an
			extension $E/k$}. These ``local'' versions of the Morel-Voevodsky stable $\aaa^1$-homotopy category $\SH(k)$ are
		analogues of local motivic categories introduced in \cite{Iso}
		but with a substantially more general notion of ``isotropy''.
		This permits to construct the, so-called, {\it isotropic Morava points} of the Balmer spectrum $\op{Spc}(\SH^c(k))$ of (the compact part of) the Morel-Voevodsky category. These
		analogues of topological Morava points are parametrized by
		the choice of Morava K-theory and a $K(p,m)$-equivalence
		class of extensions $E/k$. This
		provides a large supply of new points, and substantially improves our understanding of the spectrum. An interesting new feature is that the specialization among isotropic points
		behaves differently than in topology. 
	\end{abstract}

	\section{Introduction}
	
	The spectrum of a commutative ring $A$ permits to look at the
	category of $A$-modules geometrically. The Balmer spectrum $\op{Spec}({\mathcal D})$  - \cite{Bal-1} plays a similar role for a 
	$\otimes$-trangulated category ${\mathcal D}$. This is a ringed topological space whose points are
	{\it prime} $\otimes\op{-}\triangle$-ed ideals of ${\mathcal D}$.
	
	One of the most important and useful $\otimes$-trangulated categories is the
	stable homotopy category $\SH$ in topology. The Balmer spectrum (of the compact part) of it was essentially computed by the famous {\it Nilpotence Theorem} of
	Devinatz-Hopkins-Smith \cite{DHS}, though in a different language, see \cite[Corollary 9.5]{BalSSS} for the modern
	description. It consists of a generic point ${\frak a}_0$ 
	given by the
	prime ideal of torsion compact objects of $\SH$ (in other
	words, by objects killed by $\wedge\op{H}\qq$), which specializes
	into the chains of ordered points parametrized by prime numbers
	$p$. Points on the $p$-chain are numbered by natural numbers 
	$1\leq m\leq\infty$ (and infinity), and ${\frak a}_{p,n}$-point is a specialization of the ${\frak a}_{p,m}$-one if and only if 
	$n\geq m$. 
	The prime ideal ${\frak a}_{p,m}$ consists of objects annihilated by $\wedge K(p,m)$, where $K(p,m)$ is the
	Morava K-theory. In particular, ${\frak a}_{p,\infty}$-points are closed and the respective ideal consists of objects annihilated by $\wedge \op{H}\ff_p$. The Balmer spectrum $\op{Spc}(D(Ab)^c)$ of (the compact part of) the ``topological motivic category''
	$D(Ab)$ is isomorphic to $\op{Spec}(\zz)$, it maps injectively
	into $\op{Spc}(\SH^c)$ via the (singular complex) ``motivic'' functor
	$M:\SH^c\row D(Ab)^c$, covering the generic 
	${\frak a}_0$ and closed ${\frak a}_{p,\infty}$ points.
	So, in a sense, the ``genuinely homotopic'' life is happening
	inbetween these two extremes - over the Morava points ${\frak a}_{p,m}$, where $m$ is finite. 
	
	In the algebro-geometric context, the analogue of $\SH$ is the
	{\it stable $\aaa^1$-homotopy category} $\SH(k)$ of Morel-Voevdosky - \cite{MV}. In the current paper we study the
	Balmer spectrum of (the compact part of) it. The category $\SH(k)$ is substantially more complicated than its topological
	counterpart. This is already apparent by the comparison of the ``atomic objects'' (points) in topology and algebraic geometry.
	While in topology there is only one kind of a point, in algebraic
	geometry there are many types of them, corresponding to various
	extensions $E/k$. Below we will see that this is reflected in the 
	structure of the Balmer spectra of the respective categories.
	Balmer has shown that, in a rather general situation, the spectrum of the $\otimes\op{-}\triangle$-ed category surjects to the
	Zariski spectrum of the endomorphism ring of the unit object.
	In particular, $\op{Spc}(\SH^c(k))$ surjects to $\op{Spec}(\op{GW}(k))$ - \cite[Corollary 10.1]{BalSSS}, where $GW(k)$ is
	the {\it Grothendieck-Witt ring} of quadratic forms (which
	coincides with the $\op{End}_{\SH(k)}({\mathbbm{1}})$ by the 
	result of Morel -\cite{Mor04}). This was refined by Heller-Ormsby \cite{HO}, who showed that $\op{Spc}(\SH^c(k))$ surjects also to
	the spectrum of the {\it Milnor-Witt K-theory} of $k$ (the latter
	spectrum is bigger than that of the $GW(k)$). There is also
	the topological realisation functor (we assume everywhere that
	characteristic of our field is zero) $\SH(k)\row\SH$ which gives
	a copy of $\op{Spc}(\SH)$ inside our Balmer spectrum. But this is far from everything. Our results, in particular, show that the discrepancy between the two
	spectra is dramatic, even in terms of cardinality of them. 
	
	Our approach to $\SH(k)$ is based on the idea of {\it isotropic realisations}. Such realisations and, more generally, the 
	{\it local versions} of the category corresponding to field extensions $E/k$, in the case of the Voevodsky category of motives $\DM(k)$, were constructed in \cite{PicQ} and \cite{Iso}. In the motivic
	case, one needs to choose a prime number $p$ and then annihilate
	the motives of $p$-anisotropic varieties over $k$, that is, varieties which have no (closed) points of degree prime to $p$
	(the idea of such localisation is due to Bachmann - \cite{BQ}).
	We get the {\it isotropic motivic category} $\DM(k/k;\ff_p)$
	together with the natural functor 
	$\DM(k)\row\DM(k/k;\ff_p)$.
	One may combine this with the extension of fields and
	get functors with values in $\DM(E/E;\ff_p)$, for every extension
	$E/k$. For a general field $E$, the isotropic motivic category
	$\DM(E/E;\ff_p)$ will not be particularly simple (for example,
	for an algebraically closed field, it will coincide with the
	global category $\DM(E)$). But there is an important class of
	{\it flexible fields} - \cite[Definition 1.1]{Iso} for which this
	category is handy. These are purely transcendental extensions
	of infinite transcendence degree of some other fields. Any field
	$E$ can be embedded into its {\it flexible closure}
	$\wt{E}:=E(\pp^{\infty})$, and so, for any extension $E/k$, we
	get the {\it isotropic realisation} functor
	$\psi_{p,E}:\DM(k)\row\DM(\wt{E}/\wt{E};\ff_p)$ with values in a
	much simpler category. It was shown in \cite[Theorem 5.13]{INCHKm} that the kernels ${\frak a}_{p,E}$ of these functors
	(on the compact part) are prime $\otimes\op{-}\triangle$-ed ideals and so, provide {\it isotropic points} of the Balmer spectrum $\op{Spc}(\DM(k)^c)$. Two such points ${\frak a}_{p,E}$ and
	${\frak a}_{q,F}$ coincide if and only if $p=q$ and extensions
	$E/k$ and $F/k$ are {\it $p$-equivalent} - 
	\cite[Definition 2.2]{Iso}, and there is a huge number of such $p$-equivalence classes of extensions, in general - see \cite[Example 5.14]{INCHKm}. 
	
	In the case of $\SH(k)$, we need to use a more general version
	of {\it isotropy}. In \cite[Definition 2.1]{INCHKm},
	the notion of $A$-anisotropic varieties
	for any oriented cohomology theory $A^*$ was introduced - see
	Definition \ref{A-anis-var}. In the case of the Chow groups
	$\CH^*/p$ modulo $p$ it coincides with the $p$-anisotropy considered above. The natural theories to consider are the 
	Morava K-theories $K(p,m)^*$ and related theories $P(m)^*$.
	These {\it small} algebro-geometric analogues of the respective topological theories are {\it free theories} in the sense of Levine-Morel  \cite{LM} on the category $\smk$ of smooth varieties, obtained from the $BP^*$-theory by change of coefficients: $K(p,m)^*(X):=BP^*(X)\otimes_{BP}K(p,m)$ and
	$P(m)^*(X):=BP^*(X)\otimes_{BP}P(m)$, where $BP=\zz_{(p)}[v_1,v_2,\ldots]$ with generators $v_r$ of dimension $p^r-1$,
	$K(p,m)=\ff_p[v_m,v_m^{-1}]$, with the map sending other
	generators to zero (our constructions don't depend on the choice of these generators), and $P(m)=BP/I(m)$, where
	$I(m)=(v_0,v_1,\ldots,v_{m-1})$ is the {\it invariant ideal}
	of Landweber \cite{La73b}. We construct the {\it $K(p,m)$-isotropic} version $\SH_{(p,m)}(k/k)$ of the Morel-Voevodsky category,
	for all $p$ and $m$.

	To start with, (similar to the motivic case) we annihilate the 
	$\Sigma^{\infty}_{\pp^1}$-spectra of all $K(p,m)$-anisotropic
	varieties over $k$. And then, in the resulting category
	$\widehat{\SH}_{(p,m)}(k/k)$, we
	annihilate all compact objects on whose $MGL$-motives $v_m$ is
	nilpotent - see Definition \ref{isotr-cat-K-p-m}. Alternatively,
	$\SH_{(p,m)}(k/k)$ is the Verdier localisation of $\SH(k)$ by 
	the category generated by the compact objects whose $MGL$-motives
	belong to the localising $\otimes$-ideal generated by 
	$MGL$-motives of $K(p,m)$-anisotropic varieties and
	${\mathbbm{1}}^{MGL}/v_m$ (where we assume $v_{\infty}=1$).
	In the case $m=\infty$ and non-formally real fields, the category $\SH_{(p,\infty)}(k/k)$
	is obtained in one step as in the motivic situation and is
	closely related to $\DM(k/k;\ff_p)$. It coincides with the category considered by Tanania in \cite{ISHG} and \cite{Tan-COIMC}. We have a natural projection from ``global'' to isotropic
	category from which, as above, we obtain the {\it isotropic realisations} 
	$$
	\psi_{(p,m),E}:\SH(k)\row\SH_{(p,m)}(\wt{E}/\wt{E}), 
	$$
	for
	all extensions $E/k$. These realisations take values in isotropic
	motivic categories over flexible fields. Below we will show that
	the zero ideal of such a category is prime. Moreover, it is
	expected that the Balmer spectrum of it is a singleton. 
	One may introduce a certain $K(p,m)$-equivalence relation
	$\stackrel{(p,m)}{\sim}$
	on the set of extensions - see Section \ref{section-ishc},
	and, as appears, it is sufficient to use one representative $E$ from each
	equivalence class.
	Denote as ${\frak a}_{(p,m),E}$ the
	kernel of $\psi_{(p,m),E}$ restricted to the compact part
	$\SH^c(k)$. It is a $\otimes\op{-}\triangle$-ed ideal. 
	Our main result is Theorem \ref{Main}:
	
	\begin{theorem}
		\label{Main-intro}
		\phantom{a}\hspace{5mm}\phantom{a}
		\begin{itemize}
			\item[$(1)$] ${\frak a}_{(p,m),E}$ are prime and so, points of the Balmer spectrum $\Spc(\SH^c(k))$. 
			\item[$(2)$] ${\frak a}_{(p,m),E}={\frak a}_{(q,n),F}$ if and only if $p=q$, $m=n$ and $E/k\stackrel{(p,m)}{\sim}F/k$.
			\item[$(3)$] ${\frak a}_{(p,\infty),E}$ is the image of the point ${\frak a}_{p,E}$ of \cite[Theorem 5.13]{INCHKm} under
			the natural map of spectra $\Spc(\DM^c(k))\row\Spc(\SH^c(k))$
			induced by the motivic functor $M:\SH^c(k)\row \DM^c(k)$.
		\end{itemize}
	\end{theorem}

	We obtain a large supply of {\it isotropic points} of the Balmer
	spectrum $\op{Spc}(\SH^c(k))$. Note, that for a general field,
	the number of $(p,m)$-equivalence classes of extensions is large.
	For example, over the field of real numbers $\rr$, for every
	topological Morava point ${\frak a}_{2,m}$, $1\leq m\leq\infty$
	we have $2^{2^{\aleph_0}}$ isotropic points ${\frak a}_{(2,m),E}$ - see Example \ref{real-case}. In particular, the cardinality of $\op{Spc}(\SH^c(\rr))$ is
	equal to the cardinality of the set of all subsets of this category. 
	
	To show that the zero ideal of the Morava-isotropic category
	$\SH^c_{(p,m)}(k/k)$, over a flexible field $k$, is prime, we
	study the category $\widehat{\SH}^c_{(p,m)}(k/k)$
	which has natural projection $\widehat{M}^{MGL}$ to the
	$MGL$-motivic version $\widehat{\MGL}^c_{(p,m)}\op{-}mod$ - see
	(\ref{two}). The latter category has a natural {\it weight
		structure} in the sense of Bondarko \cite{Bon} whose heart
	can be identified with the category of 
	{\it isotropic $P(m)[\ov{x}]^*$-Chow motives} (the same generators $\ov{x}$ as in $\Omega^*_{(p)}=BP^*[\ov{x}]$).
	Here comes the crucial step.
	It was shown in \cite[Theorem 5.1]{INCHKm} that the category
	of {\it isotropic $P(m)^*$-Chow motives} (over a flexible field)
	is equivalent to the category
	of {\it numerical $P(m)^*$-Chow motives}. This identifies our
	heart with the numerical category
	$Chow^{P(m)[\ov{x}]}_{Num}(k)$ - see 
	Proposition \ref{weight-str-M}.
	The weight filtration assigns to any compact object $Y$ of $\widehat{\SH}^c_{(p,m)}(k/k)$ some choice $t(Y)$ of the 
	{\it weight complex} of it - an object of
	$K^b(Chow^{P(m)[\ov{x}]}_{Num}(k))$. The radicals of 
	$\laz$-annihilators of $Y$ and $t(Y)$ coincide, and if
	$Y$ is the $\widehat{MGL}$-motive of some 
	$U\in\widehat{\SH}^c_{(p,m)}(k)$, then the
	action of the Landweber-Novikov operations ensures that this
	radical is the invariant ideal $I(p,r)$ of Landweber, for
	some $r\geq m$ - see Proposition \ref{sq-J-U-p-r}. 
	
	The $P(m)^*$-theory 
	naturally projects to the Morava K-theory $K(p,m)^*$,
	which induces the functor 
	$K^b(Chow^{P(m)[\ov{x}]}_{Num}(k))\row K^b(Chow^{K(p,m)}_{Num}(k))$.
	But the category of numerical Morava-motives is semi-simple.
	As a result, we obtain a well-defined functor
	$t_{K(p,m)}:\widehat{\MGL}^c_{(p,m)}\op{-}mod\row 
	K^b(Chow^{K(p,m)}_{Num}(k))$. Moreover, we prove that the $\widehat{M}^{MGL}(U)$, for $U\in\widehat{\SH}^c_{(p,m)}(k/k)$, belongs to the
	kernel of it if and only if it is annihilated by 
	some power of $v_m$ - see Proposition \ref{vm-nilp-Km-triv}.
	It follows from the results of Sosnilo \cite{Sos} and
	Aoki \cite{Aok} that $t_{K(p,m)}$ is a $\otimes\op{-}\triangle$-ed
	functor. Finally, the zero ideal in $K^b(Chow^{K(p,m)}_{Num}(k))$
	is prime, since the respective $K(p,m)$-Chow category is semi-simple and the ring $K(p,m)$ is an integral domain. This proves
	part $(1)$ of the above Theorem. 
	
	To prove part $(2)$, in analogy with
	\cite[Definition 2.3]{Iso}, we introduce the {\it local}
	versions $\SH^c_{(p,m)}(E/k)$ of the Morel-Voevodsky category,
	parametrized by finitely generated extensions $E/k$ (and the 
	choice of Morava K-theory) - see 
	Definition \ref{isotr-cat-K-p-m-E-k} and then prove that the
	local-to-isotropic functor is conservative on the image
	of $\SH^c(k)$, if $k$ is flexible - Proposition \ref{D-A9}. 
	This identifies the $\stackrel{(p,m)}{\sim}$-equivalence class
	of the point ${\frak a}_{(p,m),E}$ - 
	Proposition \ref{ideals-Kpm-eq-ext}. The prime number $p$ is its
	characteristic. Finally, as in topology, the number $m$ is determined with the help of certain ``test spaces'' $X_{p,n}$ of
	{\it $p$-type} exactly $n$. Here we use 
	the construction of Mitchell \cite{Mit} adapted to the 
	algebro-geometric situation - \cite{TS}. 
	
	The equivalence relation $\stackrel{(p,m)}{\sim}$ is 
	{\it coarser} than $\stackrel{(p,n)}{\sim}$, for $m\leq n$, 
	so we get the
	natural surjection 
	$$
	{\mathcal P}_{(p,m)}\twoheadleftarrow {\mathcal P}_{(p,n)}
	$$
	between the respective sets of Morava-isotropic points. 
	Nevertheless, the expected analogue of the topological
	specialisation ${\frak a}_{p,m}\succ {\frak a}_{p,n}$ doesn't
	hold for isotropic points. The obstruction is given, for example,
	by the {\it norm-varieties} of Rost - see Remark \ref{strange-iso-spec}.
	
	The article is organized as follows. In Section \ref{two}, we recall the notion of a {\it small} oriented cohomology theory,
	introduce (generalized) {\it isotropic equivalence}, {\it flexible fields} and cite the results from \cite{INCHKm} we need.
	In Section \ref{section-ishc}, we define {\it Morava-isotropic stable homotopy categories} and {\it isotropic realisations}. 
	In Section \ref{section-weight} we introduce the needed {\it weight structure}, while Section \ref{section-LN} is devoted to
	the Landweber-Novikov operations. In Section \ref{section-nMm},
	we make a closer look at {\it numerical Morava-Chow motivic category} and introduce {\it Morava weight cohomology}. In
	Section \ref{section-MT}, we formulate our Main Theorem and prove
	parts $(1)$ and $(3)$ of it. Finally, in Section \ref{section-lsh} we introduce the {\it local stable homotopy categories}
	and prove part $(2)$ of the Main Theorem.
	
	\medskip
	
	\noindent
	{\bf Acknowledgements:}
	We are very grateful to Tom Bachmann, Paul Balmer, Martin Gallauer, John Greenlees, Daniel Isaksen and Beren Sanders
	for very helpful discussions and stimulating questions. We would like to thank the Referees for useful suggestions and remarks which improved the exposition.
	The support of the EPSRC standard grant EP/T012625/1 is gratefully acknowledged.
	
	\section{Oriented cohomology theories and isotropic equivalence}
	Everywhere in this article the ground field $k$ will be any field
	of characteristic zero. This assumption will allow us to work comfortably with oriented cohomology theories and, in particular, to use the algebraic cobordism of Levine-Morel. 
	
	Let $A^*:\smk\row Rings$ be an oriented cohomology theory with localisation (on the category of smooth varieties over $k$) in the sense of \cite[Definition 2.1]{SU} (which is the standard axioms of Levine-Morel \cite[Definition 1.1.2]{LM} plus the excision axiom $(EXCI)$). We will also call such theories {\it small} in contrast to
	{\it large} theories represented by spectra in $\SH(k)$. 
	
	The {\it algebraic cobordism} $\Omega^*$ of Levine-Morel \cite{LM}
	is the universal oriented cohomology theory. That is, for any oriented theory $A^*$, there is a unique morphism of theories (a map respecting pull-backs and push-forwards) $\theta_A:\Omega^*\row A^*$ - see \cite[Theorem 1.2.6]{LM}.
	By the result of Levine \cite{Lcomp}, the theory $\Omega^*$ coincides with the
	{\it pure part} $MGL^{2*,*}$ of the $MGL$-theory of Voevodsky.
	
	Among oriented cohomology theories, there are {\it free theories} in the sense of Levine-Morel. Such theories $A^*$ are obtained from algebraic cobordism by change of coefficients:
	$A^*=\Omega^*\otimes_{\laz}A$ and are in $1$-to-$1$ correspondence with the {\it formal group laws} $\laz\row A$. These are exactly the
	{\it theories of rational type} in the sense of \cite{SU}.
	
	Recall from \cite[Definition 2.1]{INCHKm}, that for any (small) oriented cohomology theory $A^*$, one can introduce the notion of $A$-anisotropic varieties.
	
	\begin{definition}
		\label{A-anis-var}
		Let $X\stackrel{\pi}{\row}\op{Spec}(k)$ be a smooth projective variety over $k$. Then $X$ is called $A$-anisotropic, if the
		push-forward map $\pi_*:A_*(X)\row A_*(\op{Spec}(k))$ is zero. 
	\end{definition}
	
	Such (non-empty) varieties exist only if $A^*$ is $n$-torsion, for 
	some $n\in\nn$ - see \cite[Remark 2.4(1)]{INCHKm}. In the case 
	$A^*=\CH^*/n$ of Chow groups modulo $n$, we get the usual notion of
	$n$-anisotropy.
	
	If $X$ is a smooth projective variety and $x\in A^*(X)$, we say that $x$ is {\it $A$-anisotropic}, if it is in the image of the push-forward map from some $A$-anisotropic variety. The $A^*$-anisotropic
	classes form an ideal stable under $A^*$-correspondences - see
	\cite[Proposition 2.6]{INCHKm}. Following \cite[Definition 2.7]{INCHKm} we can define:
	
	\begin{definition}
		\label{A-iso-theory}
		For $X/k$ - smooth projective, define:
		$$
		A^*_{iso}(X):=A^*(X)/(A-\text{anisotropic classes}).
		$$
	\end{definition}
	
	Applying the construction from \cite[Example 4.1]{Iso}, this extends
	to the oriented cohomology theory $A^*_{iso}$ on $\smk$ 
	(in the sense of \cite[Definition 2.1] {SU}) - the 
	{\it isotropic version} of the theory $A^*$.
	By the projection formula, any {\it anisotropic} class is {\it numerically trivial} in the sense of the pairing
	$$
	\la\,,\ra:A^*(X)\times A^*(X)\row A
	$$
	given by $\la x,y\ra:=\pi_*(x\cdot y)\in A=A^*(\op{Spec}(k))$. Thus,
	the isotropic version of the theory surjects to the numerical one:
	$A^*_{iso}\twoheadrightarrow A^*_{Num}$. 
	
	More generally, if $A^*$ and $B^*$ are two oriented cohomology theories,
	we may consider the {\it $A$-isotropic version} $B^*_{A\op{-}iso}$
	of the theory $B^*$, defined for smooth projective varieties as:
	$$
	B^*_{A\op{-}iso}(X):=B^*(X)/(A-\text{anisotropic classes}).
	$$
	The respective category of Chow motives will be denoted 
	$Chow^B_{A\op{-}iso}(k)$.
	
	As in topology, the $p$-localisation of algebraic cobordism splits
	as a polynomial algebra over the $BP^*$-theory:
	$\Omega^*_{\zz_{(p)}}=BP^*[x_i|_{i\neq p^r-1}]=BP^*[\ov{x}]$, where
	the coefficient ring of the $BP^*$-theory is
	$BP=\zz_{(p)}[v_1,v_2,\ldots]$ with $\ddim(v_i)=p^i-1$.
	On $\Omega^*$ and $BP^*$ there is the action of the Landweber-Novikov
	operations and, by the results of Landweber \cite{La73b}, the ideals
	$I(m):=(p,v_1,\ldots,v_{m-1})\subset BP$, for any $1\leq m\leq\infty$, are invariant under this action (actually, together with $I(0):=(0)$, are the only prime invariant ideals). 
	
	For any prime $p$ and any $1\leq m\leq\infty$, we can introduce the
	free theories $K(p,m)^*:=BP^*\otimes_{BP}K(p,m)$, where $K(p,m)=\ff_p[v_m,v_m^{-1}]$ (the other variables are mapped to zero),
	for $m<\infty$ and $K(p,m)=\ff_p$ (with the augmentation map), for 
	$m=\infty$. These are (small) {\it Morava K-theories}.
	We will also use the related free theories 
	$P(m)^*:=BP^*\otimes_{BP}P(m)$, where $P(m)=BP/I(m)$.
	Note that $K(p,\infty)^*=P(\infty)^*=\Ch^*=\CH^*/p$ is just Chow
	groups modulo $p$. 
	The introduced {\it small} theories $BP^*$, $K(p,m)^*$, $P(m)^*$ are {\it pure parts} of the respective {\it large} theories - see \cite[Corollary 6.3]{LT}. 
	
	The isotropy of varieties with respect to these theories may
	be determined from the invariant
	$I(X):=\op{Im}(\pi_*:BP_*(X)\row BP)$. Being an invariant (and finitely-generated, since $\Omega^*(X)$ as a $\laz$-module is generated in non-negative co-dimension by the result of Levine-Morel \cite{LM}) prime ideal of $BP$, the radical of it is either 
	$I(n)$, for $0\leq n<\infty$, or the whole $BP$.
	In \cite{INCHKm} it was shown that:
	
	\begin{proposition} {\rm (\cite[Proposition 4.15]{INCHKm})}
		\label{anis-Pm-Km-IX}
		The following conditions are equivalent:
		\begin{itemize}
			\item[$(1)$] $X$ is $P(m)$-anisotropic; 
			\item[$(2)$] $X$ is $K(m)$-anisotropic;
			\item[$(3)$] $\sqrt{I(X)}=I(n)$, for some $1\leq n\leq m$.
		\end{itemize}
	\end{proposition}

	In our considerations, an important role will be played by a special
	type of fields called {\it flexible} - \cite[Definition 1.1]{Iso}. These are
	purely transcendental extensions of infinite transcendence degree of
	some other fields. Any field $k$ can be embedded into its
	{\it flexible closure} $\wt{k}=k(\pp^{\infty})$.

	It appears that for the $K(p,m)^*$ and $P(m)^*$-theories over a flexible field, the
	isotropic versions coincide with numerical ones.
	This fundamental fact, established in \cite{INCHKm}, plays the 
	crucial role in our constructions below.
	
	\begin{theorem}
		\label{Thm-iso-num} {\rm (\cite[Theorem 1.4]{INCHKm})}
		Let $k$ be flexible. Then, for $1\leq m\leq\infty$, 
		$$
		P(m)^*_{iso}=P(m)^*_{Num}\hspace{5mm}\text{and}\hspace{5mm}
		K(p,m)^*_{iso}=K(p,m)^*_{Num}.
		$$
	\end{theorem}
	
	\section{Isotropic stable homotopy categories}
	\label{section-ishc}
	
	Let $k$ be any field of characteristic zero and $\SH(k)$ be
	the stable $\aaa^1$-homotopy category of Morel-Voevodsky \cite{MV}. Let $p$ be a prime, $1\leq m\leq\infty$ be a natural number (or infinity), and $K(p,m)^*$ be the (small) Morava K-theory introduced above. 
	
	Let $Q_{(p,m)}$ be the disjoint union of all $K(p,m)$-anisotropic
	smooth projective varieties over $k$ (up to isomorphism). Let
	$\check{C}(Q_{p,m})$ be the respective \v{C}ech simplicial scheme and ${\frak X}_{Q_{(p,m)}}$ be the $\Sigma^{\infty}_{\pp^1}$-spectrum of it. It is a $\wedge$-projector (apply \cite[Lemma 1.15]{MV} to the projection ${\frak X}\wedge{\frak X}\row{\frak X}$). Let us denote as 
	$\Upsilon_{(p,m)}$ the complementary projector $\op{Cone}({\frak X}_{Q_{(p,m)}}\row{\mathbbm{1}})$. Denote as
	$\widehat{\SH}_{(p,m)}(k/k)$ the category
	$\Upsilon_{(p,m)}\wedge \SH(k)$. It is naturally a full subcategory of 
	$\SH(k)$ equivalent to the Verdier localisation of $\SH(k)$ by the localising subcategory generated by $K(p,m)$-anisotropic varieties. The natural projection $pr:\SH(k)\row \widehat{\SH}_{(p,m)}(k/k)$ is left adjoint to the embedding, which respects direct sums. Thus, $pr$ maps compact objects to compact ones.
	
	We have an adjoint pair:
	\begin{equation}
		\label{one}
		\xymatrix{
			M^{MGL}:\SH(k) \ar@<1ex>[rr]& & \MGL\op{-}mod:B\ar@<1ex>[ll]
		},
	\end{equation}
	where $M^{MGL}$ is a $\otimes$-functor and $B$ satisfies the projection formula. Here $\MGL$ is a highly structured commutative ring spectrum - see \cite{PPR}, and the
	category of $\MGL\op{-}mod$ has an enhancement as a
	symmetric monoidal $\infty$-category.	
	If we denote $\Upsilon^{MGL}_{(p,m)}:=M^{MGL}(\Upsilon_{(p,m)})$ and $\widehat{\MGL}_{(p,m)}\op{-}mod:=\Upsilon^{MGL}_{(p,m)}\otimes\MGL\op{-}mod$, then it descends to the adjoint pair
	\begin{equation}
		\label{two}
		\xymatrix{
			\widehat{M}^{MGL}:\widehat{\SH}_{(p,m)}(k/k) \ar@<1ex>[rr]& & \widehat{\MGL}_{(p,m)}\op{-}mod:\widehat{B}\ar@<1ex>[ll]
		}.
	\end{equation}
	
	Let $\ca$ be the localising subcategory of 
	$\widehat{\SH}_{(p,m)}(k/k)$ generated by all compact objects $U$ such that: for $m<\infty$, $v_m$ is nilpotent on $\widehat{M}^{MGL}(U)$; for $m=\infty$, $\widehat{M}^{MGL}(U)=0$.
	
	\begin{definition}
		\label{isotr-cat-K-p-m}
		The isotropic stable homotopy category $\SH_{(p,m)}(k/k)$ is the Verdier localisation of the category $\widehat{\SH}_{(p,m)}(k/k)$ with respect to $\ca$. It comes equipped with the natural functor
		$\SH(k)\row \SH_{(p,m)}(k/k)$.
	\end{definition}
	
	\begin{remark}
		\label{m-inf-case}
		If $m=\infty$ and $k$ is not formally real, then by the result of Bachmann - \cite{BachCons}, the $MGL$-motivic functor $M^{MGL}:\SH(k)\row \MGL\op{-}mod$ is conservative on $\Upsilon_{(p,\infty)}\wedge \SH^c(k)$. Thus, $\ca=0$, in this case, and the category $\SH_{(p,\infty)}(k/k)$ coincides with the category considered by Tanania in \cite{ISHG} and 
		\cite{Tan-COIMC}.
	\end{remark}

	Since the subcategory $\ca$ is generated by compact objects, our
	realisation restricts to compact parts:
	$\SH^c(k)\row\SH_{(p,m)}^c(k/k)$ - see \cite[Lemma 4.4.4]{Nee}.
	
	\begin{proposition}
		\label{altern-def-isotr-sh}
		Let $1\leq m <\infty$. Then $\SH_{(p,m)}(k/k)$ is the Verdier localisation of $\SH(k)$ by the subcategory, generated by compact objects $U$ such that $M^{MGL}(U)$ belongs to the localising $\otimes$-ideal of $\MGL\op{-}mod$ generated
		by ${\mathbbm{1}}^{MGL}/v_m:=\op{Cone}({\mathbbm{1}}^{MGL}(*)[2*]\stackrel{v_m}{\lrow}{\mathbbm{1}}^{MGL})$ and $MGL$-motives of $K(p,m)$-anisotropic varieties.
	\end{proposition}
	
	\begin{proof}
		Let $\hi_{(p,m)}^{MGL}:=M^{MGL}({\frak X}_{Q_{(p,m)}})$ be the projector complementary to $\Upsilon^{MGL}_{(p,m)}$.
		
		$(\rightarrow)$ Let $U\in \SH^c(k)$ be a compact object, such that
		$v_m$ is nilpotent on $\Upsilon^{MGL}_{(p,m)}\otimes M^{MGL}(U)$.
		Then by \cite[Theorem 2.15]{BalSSS}, this object belongs to the thick
		$\otimes\op{-}\triangle$-ed ideal of $MGL\op{-}mod$ generated by 
		$\Upsilon^{MGL}_{(p,m)}\otimes {\mathbbm{1}}^{MGL}/v_m$.
		But 
		$M^{MGL}(U)$ is an extension of $\hi_{(p,m)}^{MGL}\otimes M^{MGL}(U)$ and $\Upsilon^{MGL}_{(p,m)}\otimes M^{MGL}(U)$, where
		$\hi_{(p,m)}^{MGL}$ belongs to the localising $\otimes$-ideal generated by $MGL$-motives of $K(p,m)$-anisotropic varieties.
		Hence, $M^{MGL}(U)$ belongs to the specified ideal.
		
		$(\leftarrow)$ If $U\in \SH^c(k)$ be a compact object such that $M^{MGL}(U)$ belongs
		to the ideal described, then $\Upsilon^{MGL}_{(p,m)}\otimes M^{MGL}(U)$ belongs to the localising $\otimes$-ideal of 
		$\widehat{\MGL}_{(p,m)}\op{-}mod=\Upsilon^{MGL}_{(p,m)}\otimes \MGL\op{-}mod$
		generated by $\Upsilon^{MGL}_{(p,m)}\otimes{\mathbbm{1}}^{MGL}/v_m$.
		And this object is still compact. Hence, the action of $v_m$ on it
		is nilpotent, again by \cite[Theorem 2.15]{BalSSS}.
		\Qed
	\end{proof}
	
	The category $\SH_{(p,m)}(k/k)$ is a $\otimes-\triangle$-category and the natural localisation $\SH(k)\row \SH_{(p,m)}(k/k)$
	is a tensor triangulated functor.\\

	We can introduce the partial $K(p,m)$-order on the set of field extensions of $k$. 
	Let $E=\op{colim}_{\alpha}E_{\alpha}$, and $F=\op{colim}_{\beta}F_{\beta}$, for some finitely generated extensions $E_{\alpha}=k(Q_{\alpha})$ and $F_{\beta}=k(P_{\beta})$ with smooth models
	$Q_{\alpha}$ and $P_{\beta}$.
	We say that 
	$E/k\stackrel{(p,m)}{\geq}F/k$, if for any $\beta$, there exists $\alpha$ such that $(P_{\beta})_{E_{\alpha}}$ is $K(p,m)$-isotropic.
	This is equivalent to the surjectivity of the push-forward map
	$\pi_*:K(p,m)_*(Q_{\alpha}\times P_{\beta})\twoheadrightarrow K(p,m)_*(Q_{\alpha})$.
	Indeed, the Morava K-theory $K(p,m)^*$ is a {\it free} theory, and so, a {\it coherent} theory - see \cite[Definition 2.6, Corollary 2.13]{RNCT}. Thus, 
	$K(p,m)^*((P_{\beta})_{E_{\alpha}})$ is the colimit of $K(p,m)^*(U\times P_{\beta})$, where $U$ runs over
	non-empty open suschemes of $Q_{\alpha}$.
	The first condition then is equivalent to the fact that
	the push-forward $K(p,m)_*(U\times P_{\beta})\row K(p,m)_*(U)$ is surjective, for some such $U$.  
	As the pull-back map $K(p,m)^*(Q_{\alpha}\times P_{\beta})\row K(p,m)^*(U\times P_{\beta})$ is surjective as well (by the $(EXCI)$ axiom), there is $y\in K(p,m)_*(Q_{\alpha}\times P_{\beta})$, such that $\pi_*(y)-1=n\in K(p,m)_*(Q_{\alpha})$ is supported in positive co-dimension and so, is nilpotent. The projection formula then easily implies the surjectivity of $\pi_*$. The opposite implication is clear. 
	
	We say that two extensions are equivalent $E/k\stackrel{(p,m)}{\sim}F/k$, if $E/k\stackrel{(p,m)}{\geq}F/k$
	and $F/k\stackrel{(p,m)}{\geq}E/k$.
	
	Let $\wt{E}$ be the {\it flexible closure} of the field $E$ 
	(that is, $\wt{E}=E(\pp^{\infty})$). 
	Combining natural restrictions $\SH(k)\row \SH(\wt{E})$ with projections to isotropic categories,
	we get a family of {\it isotropic realisations}
	$$
	\psi_{(p,m),E}:\SH(k)\row \SH_{(p,m)}(\wt{E}/\wt{E}),
	$$
	where $p$ is a prime number, $1\leq m\leq\infty$, and $E/k$ runs over
	equivalence classes of the field extensions under the relation $\stackrel{(p,m)}{\sim}$ above. In 
	Theorem \ref{Main} we will show that the kernels of isotropic realisations corresponding to $(p,m)$-equivalent field extensions coincide.

	\section{The weight structure on $\op{MGL}\op{-}mod$}
	\label{section-weight}
	
	On $\MGL^c\op{-}mod$ we have the weight structure in the sense of Bondarko \cite{Bon} with the {\it heart} ${\mathcal H}$ consisting of the direct
	summands of $\MGL$-motives of smooth projective varieties. That is, ${\mathcal H}=Chow^{\Omega}(k)$.
	A similar structure exists on $\widehat{\MGL}_{(p,m)}\op{-}mod$.

	\begin{proposition}
		\label{weight-str-MGL-p-m}
		On the compact part $\widehat{\MGL}_{(p,m)}^c\op{-}mod$ we have a bounded, non-degenerate weight structure whose heart ${\mathcal H}$ 
		is equivalent to the category $Chow^{\Omega}_{K(p,m)\op{-}iso}(k)$ of 
		$K(p,m)$-isotropic $\Omega$-Chow motives.
	\end{proposition}
	
	\begin{proof} Our weight structure is determined by its heart ${\mathcal H}$ which consists of direct summands of $\widehat{M}^{MGL}$-motives of smooth projective varieties. The fact that ${\mathcal H}$ indeed defines a weight structure follows from the following result.
		
		\begin{lemma}
			\label{weight-str-l-o}
			For $X,Y\in SmProj/k$, with $\ddim(Y)=d$ and $l\in\zz$, we have:
			\begin{equation*}
				\begin{split}
					(1)\hspace{5mm}\Hom_{\widehat{\MGL}_{(p,m)}}&(\widehat{M}^{MGL}(X)(l)[2l],\widehat{M}^{MGL}(Y))\\
					&=\Omega^{d-l}(X\times Y)/(K(p,m)-\text{anisotropic classes});\\
					(2)\hspace{5mm}\Hom_{\widehat{\MGL}_{(p,m)}}&(\widehat{M}^{MGL}(X)(l)[2l],\widehat{M}^{MGL}(Y)[i])=0,\hspace{2mm}\text{for}\,\,\,i>0.
				\end{split}
			\end{equation*}
		\end{lemma}
		
		\begin{proof}
			Due to duality for $\MGL\op{-}mod$, it is sufficient to treat the
			case $X=\op{Spec}(k)$. The argument closely follows that of \cite[Proposition 2.16]{Iso}. For a smooth $Q$, denote
			as $\hi_Q$ the $MGL$-motive of the respective \v{C}ech simplicial scheme, and as $\hit_Q$ the complementary projector
			$\op{Cone}(\hi_Q\row T)$, where $T={\mathbbm{1}}^{MGL}$ is the trivial $MGL$-Tate-motive. Then the hom group
			$\Hom_{\widehat{\MGL}_{(p,m)}}(T(l)[2l],\widehat{M}^{MGL}(Y))$ is 
			given by
			$$
			\Hom_{\MGL}(\hit(l)[2l],M^{MGL}(Y)\otimes\hit)=
			\Hom_{\MGL}(T(l)[2l],M^{MGL}(Y)\otimes\hit),
			$$
			where $\hit=\hit_{Q_{(p,m)}}$ and
			$Q_{(p,m)}$ is the disjoint union of all $K(p,m)$-anisotropic varieties over $k$.
			Let $(\hi_Q)_{\geq 1}=\op{Cone}(M^{MGL}(Q)\row\hi_Q)$. This object is an extension of $M^{MGL}(Q^{\times (i+1)})[i]$, for 
			$i\geq 1$. So, there are no $\Hom$s in $\MGL\op{-}mod$ from $T(l)[2l]$ to $M^{MGL}(Y)\otimes(\hi_Q)_{\geq 1}$ (and any positive shift $[k]$ of it), since $\MGL^{a,b}(Z)=0$, for $a>2b$ and
			any smooth variety $Z$. Because $\hit_Q$ is an extension of
			$(\hi_Q)_{\geq 1}[1]$ by $\op{Cone}(M^{MGL}(Q)\row T)$, we get:
			$$
			\Hom_{\MGL}(T(l)[2l],M^{MGL}(Y)\otimes\hit_{Q_{(p,m)}})=
			\op{Coker}(\Omega_l(Y\times Q_{(p,m)})\stackrel{\pi_*}{\lrow}
			\Omega_l(Y)).
			$$
			In other words, our group is obtained from $\Omega_l(Y)$ by moding-out {\it $K(p,m)$-anisotropic classes}.
			
			The vanishing of
			$\displaystyle\Hom_{\widehat{\MGL}_{(p,m)}}(\widehat{M}^{MGL}(X)(l)[2l],\widehat{M}^{MGL}(Y)[i])$, for positive $i$, follows from
			the same arguments.
			\Qed
		\end{proof}
		
		By \cite[Theorem 4.3.2, Proposition 5.2.2]{Bon},
		we have a unique bounded non-degenerate weight structure on
		$\widehat{\MGL}_{(p,m)}\op{-}mod$ whose heart ${\mathcal H}$ is the
		category of {\it $K(p,m)$-isotropic $\Omega$-Chow-motives} 
		$Chow^{\Omega}_{K(p,m)\op{-}iso}(k)$ which is the idempotent completion of the category of correspondences, where $\Hom$s are given by
		part $(1)$ of Lemma \ref{weight-str-l-o}.
		\Qed
	\end{proof}
	
	Recall, that $\Omega^*\otimes_{\zz}\zz_{(p)}=BP^*[x_i|_{i\neq p^s-1}]=BP^*[\ov{x}]$. The results of \cite{INCHKm}
	permit to
	identify, in the case of a flexible field, the heart of our weight structure with the category of certain numerical Chow motives, which is the crucial step of the article.

	\begin{proposition}
		\label{weight-str-M}
		Let $k$ be flexible. Then
		on the compact part $\widehat{\MGL}_{(p,m)}^c\op{-}mod$ we have a bounded, non-degenerate weight structure whose heart ${\mathcal H}$ 
		is equivalent to the category $Chow^{P(m)[\ov{x}]}_{Num}(k)$ of 
		numerical $P(m)[\ov{x}]$-Chow motives.
	\end{proposition}
	
	\begin{proof}
		If $k$ is flexible, then there is a $K(p,m)$-anisotropic variety
		$Q$ whose class in the Lazard ring is $v_{m-1}$ (for example, one may use a norm-variety corresponding to some non-zero pure $m$-symbol modulo $p$). Then $I(Q)=\op{Image}(\Omega_*(Q)\stackrel{\pi_*}{\lrow}\laz)$ contains $v_i$, for $0\leq i<m$ (since this ideal
		is stable under Landweber-Novikov operations) - see 
		Example \ref{exa-norm-var}. Thus, classes
		$v_i\in\Omega^*(\op{Spec}(k))$, for $0\leq i<m$, are $K(p,m)$-anisotropic. In particular, $p\in\laz_0$ is $K(p,m)$-anisotropic.
		Since $\Omega^*\otimes_{\zz}\zz_{(p)}=BP^*[\ov{x}]$ and
		$BP^*/(v_0,\ldots,v_{m-1})=P(m)^*$, we get that, 
		for any smooth projective variety $X$,
		$$
		\Omega^*(X)/(K(p,m)\op{-}\text{anis. classes})=
		P(m)[\ov{x}]^*(X)/(K(p,m)\op{-}\text{anis. classes}).
		$$
		Since $K(p,m)$-anisotropic varieties are exactly $P(m)$-anisotropic ones, by Proposition \ref{anis-Pm-Km-IX}, the latter group can be rewritten as
		$$
		P(m)[\ov{x}]^*(X)/(P(m)\op{-}\text{anis. classes})=
		P(m)[\ov{x}]^*_{iso}(X).
		$$
		Thus, our heart is the category of isotropic $P(m)[\ov{x}]$-Chow
		motives. Finally, it follows from \cite[Theorem 5.1]{INCHKm}
		- see Theorem \ref{Thm-iso-num}
		that this category is equivalent to the category of numerical
		$P(m)[\ov{x}]$-Chow motives $Chow^{P(m)[\ov{x}]}_{Num}(k)$. 
		\Qed
	\end{proof}

	\section{Landweber-Novikov operations}
	\label{section-LN}
	
	Let $\MGL=B({\mathbbm{1}}^{MGL})$ be the $\MGL$-spectrum (see (\ref{one})), then for variables
	$b_i$ of degree $(i)[2i]$ and any monomial $\ov{b}^{\ov{r}}$ in them of degree $(s)[2s]$, we have the Landweber-Novikov map in $\SH(k)$:
	$$
	S_{L-N}^{\ov{b}^{\ov{r}}}: \MGL\lrow \Sigma^{2s,s}\MGL
	$$
	which due to the adjoint pair (\ref{one}) induces the 
	Landweber-Novikov operation
	$$
	S_{L-N}^{\ov{b}^{\ov{r}}}: MGL^{a,b}(U)\lrow MGL^{a+2s,b+s}(U),
	$$
	for any $U\in \SH(k)$ functorial on $U$. Taken together, these operations define the action of the Landweber-Novikov algebra
	$MU_*MU$ on $\SH(k)$.
	
	The natural functor $\SH(k)\lrow\widehat{\SH}_{(p,m)}(k/k)$ translates the above maps to maps  
	$$
	\widehat{S}_{L-N}^{\ov{b}^{\ov{r}}}: \widehat{\MGL}\lrow \Sigma^{2s,s}\widehat{\MGL}
	$$
	in $\widehat{\SH}_{(p,m)}(k/k)$ which induce the action
	$$
	\widehat{S}_{L-N}^{\ov{b}^{\ov{r}}}: \widehat{MGL}^{a,b}(U)\lrow \widehat{MGL}^{a+2s,b+s}(U)
	$$
	of the Landweber-Novikov algebra on $\widehat{\SH}_{(p,m)}(k/k)$.
	For compact $U$, these operations may be combined into a {\it Total} one
	$$
	\widehat{S}_{L-N}^{Tot}=\sum_{\ov{r}}\widehat{S}_{L-N}^{\ov{b}^{\ov{r}}}: \widehat{MGL}^{a,b}(U)\lrow \bigoplus_{s\in\zz}\widehat{MGL}^{a+2s,b+s}(U)
	$$
	which is multiplicative.
	
	If $V$ is compact, the group $\widehat{MGL}^{a,b}(U\otimes V^{\vee})$ may be identified with 
	$$
	\Hom_{\widehat{\MGL}_{(p,m)}}(\widehat{M}^{MGL}(U),\widehat{M}^{MGL}(V)(b)[a]). 
	$$
	In particular, for a compact $U$, we obtain the action 
	of the Landweber-Novikov algebra on
	$$
	\End(\widehat{M}^{MGL}(U))=\oplus_{i\in\zz}
	\Hom_{\widehat{\MGL}_{(p,m)}}(\widehat{M}^{MGL}(U),\widehat{M}^{MGL}(U)(i)[2i]), 
	$$
	where the action of the Total Landweber-Novikov operation $\widehat{S}_{L-N}^{Tot}$ is compartible with the composition of endomorphisms.
	
	\begin{definition}
		\label{J}
		For $U\in\widehat{\SH}_{(p,m)}^c(k)$, define $J(U)\subset\laz$ as
		$Ann(\widehat{M}^{MGL}(U))$ (cf. \cite[Definition 2.1]{Tor}).
	\end{definition}
	
	Since $\widehat{S}_{L-N}^{Tot}$ is multiplicative and fixes the identity map of $\widehat{M}^{MGL}(U)$
	in $\widehat{MGL}^{0,0}(U\otimes U^{\vee})$, it follows that $J(U)$ is 
	an invariant ideal of the Lazard ring. Hence, so is its radical
	$\sqrt{J(U)}$.
	
	\begin{proposition}
		\label{J-f-g}
		The ideal $\sqrt{J(U)}$ is finitely generated.
	\end{proposition}
	
	\begin{proof}
		Let $M$ be an arbitrary object of $\widehat{\MGL}^c_{(p,m)}\op{-}mod$ and 
		$t(M)$ be any choice of the weight complex
		of $M$ with respect to the weight structure
		from Proposition \ref{weight-str-M}. 
		It is a finite complex 
		$$
		\ldots\lrow X_{i+1}\stackrel{d_i}{\lrow}X_i\lrow\ldots,
		$$
		where $X_i$s are numerical $P(m)[\ov{x}]$-Chow motives.
		The endomorphism ring
		$$
		\End_{K^b({\mathcal H})}(t(M))=
		\oplus_{l\in\zz}\Hom_{K^b({\mathcal H})}(t(M), t(M)(l)[2l])
		$$ 
		is the cohomology of the complex $(E,d)=E_{1}\row E_0\row E_{-1}$,
		where $E_j=\oplus_i\HHom_{Chow^{P(m)[\ov{x}]}_{Num}(k)}(X_i,X_{i+j})$
		(the sum is finite), where
		$$
		\HHom(A,B)=\oplus_{l\in\zz}\Hom(A,B(l)[2l]),
		$$
		and differentials are induced by $d_i$s. 
		Since the numerical pairing is defined on the level of topological realisation, 
		$E_j$s here are sub-quotients of the respective topological cohomology generated in (co-dimensional) degrees bounded from below (as $\Omega^*$ as a $\laz$-module is generated in non-negative co-dimension, by the result of Levine-Morel \cite{LM}), and so, are finitely generated $\laz$-modules. Numerical pairing imposes finitely many linear relations on the generators of these modules, whose coefficients involve only finitely many of the polynomial generators of the Lazard
		ring. Since the ring $\laz_{\leq r}=\zz[x_1,\ldots,x_r]$ is Noetherian, the modules $E_j$ are finitely presented. The same applies to the differentials of our complex and so,
		$(E,d)=(\hat{E},\hat{d})\otimes_{\laz_{\leq r}}\laz$, for some $r$.
		Since the ring $\laz_{\leq r}$ is Noetherian,
		$\End_{K^b({\mathcal H})}(t(M))$ is a finitely
		presented $\laz$-module. Hence, $Ann(t(M))$
		is a finitely generated ideal of $\laz$. 
		Then, so is the radical $\sqrt{Ann(t(M))}$ of it.
		But by \cite[Theorem 3.3.1.II]{Bon},
		$\sqrt{Ann(M)}=\sqrt{Ann(t(M))}$. Thus, $\sqrt{Ann(M)}$ is finitely generated for any compact object $M$ of 
		$\widehat{\MGL}_{(p,m)}\op{-}mod$. It remains to apply it to
		$M=\widehat{M}^{MGL}(U)$.
                \Qed
	\end{proof}
	
	Thus, $\sqrt{J(U)}$ is a finitely generated radical invariant
	ideal of $\laz$. Since, over a flexible field, the classes $v_i,\,0\leq i<m$ are $K(p,m)$-anisotropic, $J(U)$ contains $I(p,m)=(v_0,\ldots,v_{m-1})$.
	From the results of Landweber \cite[Theorem 2.7, Proposition 3.4]{La73b} we obtain that our radical is an invariant prime ideal
	of Landweber.
	
	\begin{proposition}
		\label{sq-J-U-p-r}
		Let $k$ be flexible. Then, for any compact $U\in\widehat{\SH}_{(p,m)}(k/k)$, we have:
		$\sqrt{J(U)}=I(p,r)$, for some $r\geq m$.
	\end{proposition}
	
	\section{Numerical Morava motives}
	\label{section-nMm}
	
	We have the natural ring homomorphism $P(m)[\ov{x}]\row K(p,m)$
	of coefficient rings ($v_m$ is inverted, while $v_j$, for $j>m$
	and all $x_i$s are mapped to zero). It gives the morphism of
	{\it free} theories $P(m)[\ov{x}]^*\row K(p,m)^*$ which, in turn,
	induces the $\otimes-\triangle$-functor 
	$$
	K^b(Chow^{P(m)[\ov{x}]}_{Num}(k))\lrow
	K^b(Chow^{K(p,m)}_{Num}(k))
	$$
	of homotopy categories of the respective numerical Chow motives.
	
	Since the coefficient ring $K(p,m)=\ff_p[v_m,v_m^{-1}]$ is a {\it graded field}, the numerical Chow category
	$Chow^{K(p,m)}_{Num}(k)$ is semi-simple.
	
	\begin{proposition}
		\label{semi-simpl-Kpm}
		Any map $f:X\row Y$ in $Chow^{K(p,m)}_{Num}(k)$ can be written
		as $id_W\oplus 0$, where $X=W\oplus X'$ and $Y=W\oplus Y'$.
	\end{proposition}
	
	\begin{proof}
	    The arguments are identical to the case of $K(p,\infty)^*=\Ch^*$ treated in \cite[Lemma 5.9]{INCHKm}. The only facts about $K(p,m)^*$ we use are: 1) the degree pairing there is defined on the level of topological realisation, 2) the topological Morava K-theory satisfies the Kunneth formula, and 3) the coefficient ring $K(p,m)$ of our theory is a graded field whose zero-degree component is finite. These replace the respective facts about $\Ch^*$ used in loc. cit..
		\Qed
	\end{proof}
	
	Proposition \ref{semi-simpl-Kpm} shows that the
	category $Chow^{K(p,m)}_{Num}(k)$ is semi-simple abelian. This immediately implies (see \cite[Proposition 5.10]{INCHKm}):
	
	\begin{proposition}
		\label{zero-dif}
		Any object of $K^b(Chow^{K(p,m)}_{Num}(k))$ can be presented by a complex with zero differentials. Such a presentation is unique (up to a canonical isomorphism).
	\end{proposition}
	
	The graded pieces of the canonical representative from Proposition \ref{zero-dif} provide cohomology functors
	$$
	H^i: K^b(Chow^{K(p,m)}_{Num}(k))\row Chow^{K(p,m)}_{Num}(k).
	$$
	
	If ${\mathcal D}$ is a triangulated category with the weight structure 
	$\omega$ with the heart ${\mathcal H}$, then the {\it weight complex}
	$t(X)\in K({\mathcal H})$, for an object $X\in{\mathcal D}$, is not uniquely defined, in general. But Bondarko has shown \cite[Section 3]{Bon} that it is well defined in some quotient-category.
	Namely, for $U,V\in K({\mathcal H})$ we may consider the subgroup
	of morphisms $Z(U,V)\subset\Hom_{K({\mathcal H})}(U,V)$ that can be
	presented as $(d^V_{i-1}\circ t_i)$, for some $t_i:U_i\row V_{i-1}$. Then $Z$ is a two-sided ideal and one may define the quotient
	additive category $K_{{\frak w}}({\mathcal H})=K({\mathcal H})/Z$ - see \cite[Definition 3.1.6]{Bon} (note though, that this category is usually not triangulated). Then the assignment $X\mapsto t(X)$ gives an additive functor ${\mathcal D}\row K_{{\frak w}}({\mathcal H})$ - see \cite[Theorem 3.2.2(II)]{Bon}. 
	
	In particular, we have this {\it weak weight complex functor}
	in our situation:
	$$
	t_{{\frak w}}:\widehat{\MGL}^c_{(p,m)}\op{-}mod\lrow K^b_{{\frak w}}(Chow^{P(m)[\ov{x}]}_{Num}(k)).
	$$
	It is, actually, defined for any field $k$, as we may always
	project from $P(m)[\ov{x}]_{iso}$ to $P(m)[\ov{x}]_{Num}$.
	We may combine it with the natural functor
	$$
	K^b_{{\frak w}}(Chow^{P(m)[\ov{x}]}_{Num}(k))\lrow
	K^b_{{\frak w}}(Chow^{K(p,m)}_{Num}(k)).
	$$
	But it follows from Proposition \ref{zero-dif} that the natural
	projection 
	$$
	K^b(Chow^{K(p,m)}_{Num}(k))\row K^b_{{\frak w}}(Chow^{K(p,m)}_{Num}(k))
	$$ 
	is an equivalence of categories. Thus, we get the 
	{\it Morava weight complex functor}
	$$
	t_{K(p,m)}: \widehat{\MGL}^c_{(p,m)}\op{-}mod\lrow K^b(Chow^{K(p,m)}_{Num}(k)).
	$$
	
	Since the category $\MGL\op{-}mod$ has an $\infty$-categorical (stable) enhancement, by the result of Sosnilo \cite[Corollary 3.5]{Sos}, we
	have the {\it strong weight complex functor}
	$$
	t:\MGL^c\op{-}mod\row K^b(Chow^{\Omega}(k)),
	$$
	which restricts to the weak one when projected to $K^b_{{\frak w}}$.
	Since the category $\widehat{\MGL
	}^c_{(p,m)}\op{-}mod$ is obtained from
	$\MGL\op{-}mod$ by a $\otimes$-projector, which respects the weight structure, we get that $t$ agrees with $t_{K(p,m)}$ and so, the
	latter functor is triangulated.
	
	\begin{definition}
		\label{Morava-weight-coh} 
		Define the Morava weight cohomology 
		$$
		H^i_{K(p,m)}:\widehat{\MGL}^c_{(p,m)}\op{-}mod\lrow Chow^{K(p,m)}_{Num}(k)
		$$
		as the composition $H^i\circ t_{K(p,m)}$, for $i\in\zz$. 
	\end{definition}
	
	\begin{remark}
		Combining with the projection 
		$\MGL^c\op{-}mod\row\widehat{\MGL}^c_{(p,m)}\op{-}mod$, we get the ``global'' version of the Morava weight cohomology:
		$$
		H^i_{K(p,m)}:\MGL^c\op{-}mod\lrow Chow^{K(p,m)}_{Num}(k).
		$$
		These are cohomological functors on the triangulated category of compact $\MGL$-modules.
		\Red
	\end{remark}
	
	By the result of Aoki \cite[Corollary 4.5]{Aok}, the strong weight
	complex functor $t$ respects $\otimes$. Hence, the {\it Morava weight complex functor} 
	$t_{K(p,m)}:\widehat{\MGL}^c_{(p,m)}\op{-}mod\lrow K^b(Chow^{K(p,m)}_{Num}(k))$ is tensor triangulated. Thus, we get the Kunneth formula for {\it Morava weight cohomology}:
	
	\begin{proposition}
		\label{Kunneth-Morava} {\rm (Kunneth formula)}
		The Morava weight cohomology satisfy:
		$$
		H^k_{K(p,m)}(X\otimes Y)=\bigoplus\limits_{i+j=k}H^i_{K(p,m)}(X)\otimes H^j_{K(p,m)}(Y),
		$$
		for any compact $\MGL$-modules $X$ and $Y$.
	\end{proposition}
	
	Let $A^*$ be an oriented cohomology theory whose coefficient ring $A$ is an integral domain. Then $\otimes$ has no zero-divisors in
	$Chow^A_{Num}(k)$ on objects and morphisms - see \cite[Proposition 5.2]{INCHKm}. In particular, we get:
	
	\begin{proposition}
		\label{Kpm-Num-zd}
		The zero ideal $(0)$ of $K^b(Chow^{K(p,m)}_{Num}(k))$ is prime triangulated.
	\end{proposition}
	
	\begin{proof}
		Indeed, by Proposition \ref{zero-dif}, any object $U$
		of $K^b(Chow^{K(p,m)}_{Num}(k))$ is isomorphic to the direct sum
		of its (shifted) weight cohomology $H^i(U)[i]\in Chow^{K(p,m)}_{Num}(k)[i]$.
		But $K(p,m)=\ff_p[v_m,v_m^{-1}]$ is an integral domain, thus
		$Chow^{K(p,m)}_{Num}(k)$ has no $\otimes$-zero-divisors. Hence,
		so does $K^b(Chow^{K(p,m)}_{Num}(k))$.\\
		\phantom{a}\hspace{5mm}\Qed
	\end{proof}
	
	\section{The Main theorem}
	\label{section-MT}
	
	Denote as $U\mapsto\ov{U}$ the natural 
	$\otimes\op{-}\triangle$-ed functor
	$$
	K^b(Chow^{P(m)[\ov{x}]}_{Num}(k))\lrow
	K^b(Chow^{K(p,m)}_{Num}(k)).
	$$
	Let ${\frak b}$ be the kernel of it. It is a $\otimes$-triangulated ideal in $K^b(Chow^{P(m)[\ov{x}]}_{Num}(k))$.
	
	\begin{proposition}
		\label{Ann-b}
		Let $U\in{\frak b}$. Then $Ann(U)\subset P(m)[\ov{x}]$ is
		non-zero.
	\end{proposition}
	
	\begin{proof}
		If $m=\infty$, then the above functor is an equivalence and there is nothing to prove. Below we will assume that $m$ is finite.
		For any $M_1,M_2\in Chow^{P(m)[\ov{x}]}_{Num}(k)$, the
		$$
		\HHom_{Chow^{P(m)[\ov{x}]}_{Num}(k)}(M_1,M_2)=\oplus_l
		\Hom_{Chow^{P(m)[\ov{x}]}_{Num}(k)}(M_1,M_2(l)[2l])
		$$ 
		is a finitely
		presented $P(m)[\ov{x}]$-module. Our object $U$ is given by some
		finite complex $\ldots\row N_i\stackrel{d_i}{\row}N_{i+1}\row\ldots$. The endomorphism ring $\End(U)$ can be identified with the $0$-th cohomology of the (finite) complex
		$$
		\ldots\stackrel{\partial}{\row}F_1\stackrel{\partial}{\row}F_0\stackrel{\partial}{\row}F_{-1}\stackrel{\partial}{\row}\ldots,
		\hspace{2mm}\text{where}\hspace{2mm}F_j=\oplus_i\HHom_{Chow^{P(m)[\ov{x}]}_{Num}(k)}(N_i,N_{i+j})
		$$ 
		are finitely presentated $P(m)[\ov{x}]$-modules
		and differentials are induced by $d_i$s.
		Moreover, $(F,\partial)=(\hat{F},\hat{\partial})\otimes_{\hat{P}}P(m)[\ov{x}]$, where $\hat{P}=\ff_p[v_m,\ldots,v_n][x_i,i\leq l]$, for 
		some $n$ and $l$, and $P(m)[\ov{x}]$ is free over $\hat{P}$. In particular, since $\hat{P}$ is Noetherian, the cohomology $H_j$ of
		our complex are finitely presented $P(m)[\ov{x}]$-modules. Note, that, for a smooth projective variety $X$, any $K(p,m)$-numerically trivial class, multiplied by some power of $v_m$, can be lifted to a $P(m)$-numerically trivial class, by \cite[Proposition 4.14]{INCHKm}. Hence,
		\begin{equation}
			\label{Homs-eq}
			\HHom_{Chow^{K(p,m)}_{Num}(k)}(\ov{M}_1,\ov{M}_2)=
			\HHom_{Chow^{P(m)[\ov{x}]}_{Num}(k)}(M_1,M_2)\otimes_{P(m)[\ov{x}]}K(p,m).
		\end{equation}
		Here $\HHom_{Chow^{P(m)[\ov{x}]}_{Num}(k)}(M_1,M_2)$ is given by 
		the numerical $P(m)[\ov{x}]$ of a direct summand of some smooth projective variety. There is a natural filtration by codimension
		of support on this $\laz$-module. The graded pieces of it possess the action of the Landweber-Novikov algebra $MU_*MU$. Since the respective $\laz$-modules are finitely presented, by the result of Landweber \cite[Lemma 3.3]{La73b}, these are extensions of (finitely many) 
		$\laz$-modules $\laz/I(q,n)$, and so, our $\HHom$ is an extension of such modules. Moreover, since it is a $P(m)$-module, it should be an extension of modules of the type
		$\laz/I(p,n)=P(n)[\ov{x}]$, with $n\geq m$.
		
		Let $P\{m\}^*$ be the free theory with the coefficient ring
		$P\{m\}=P(m)[v_m^{-1}]$. Note that the functor 
		$\otimes_{P(m)}P\{m\}$ is exact and it annihilates the modules
		$P(n)$, for $n>m$. Thus, 
		$$
		\HHom_{Chow^{P(m)[\ov{x}]}_{Num}(k)}(M_1,M_2)\otimes_{P(m)}P\{m\}
		$$ 
		is a free $P\{m\}[\ov{x}]$-module
		of finite rank. This shows that our complex $(F,\partial)$ is
		adjusted for the right exact functor $\otimes_{P(m)[\ov{x}]}K(p,m)$.
		
		We need the following facts. The first one is a variant of 
		\cite[Lem. 3.4]{TS}.
		
		\begin{lemma}
			\label{Pm-Km-l1}
			Let $L$ be a finitely-generated $P(m)[\ov{x}]$-module, such that\\ 
			$L\otimes_{P(m)[\ov{x}]}K(p,m)=0$. Then $Ann(L)\subset P(m)[\ov{x}]$ is non-zero. Moreover, $Ann(L)\otimes_{P(m)[\ov{x}]}K(p,m)=K(p,m)$.
		\end{lemma}
		
		\begin{proof}
			Denote as $I\subset P(m)[\ov{x}]$ the ideal generated by $v_r$, $r>m$ and all $x_i$, $i\neq p^s-1$. Then the proof is identical to that of \cite[Lemma 3.4]{TS} (which works for any polynomial algebra).
			\Qed
		\end{proof}
		
		\begin{lemma}
			\label{Pm-Km-l2}
			Let $L$ be a finitely-generated $P(m)[\ov{x}]$-module, such that\\
			$L\otimes_{P(m)[\ov{x}]}K(p,m)=0$. Then 
			$$\op{Tor}^{P(m)[\ov{x}]}_i(L,K(p,m))=0,\hspace{2mm}\text{for all}\,\, i.
			$$
		\end{lemma}
		
		\begin{proof}
			Here $\op{Tor}^{P(m)[\ov{x}]}_i(L,K(p,m))$ is a $K(p,m)$-module which is
			annihilated by every $u\in Ann(L)$. Since $Ann(L)\otimes_{P(m)[\ov{x}]}K(p,m)=K(p,m)$, this module is zero.
			\Qed
		\end{proof}
		
		Let $H_j$ be the $j$-cohomology of the complex $(F,\partial)$. 
		These are finitely-presented $P(m)[\ov{x}]$-modules.
		Let us show by the increasing induction on $l$ that 
		$\op{Tor}^{P(m)[\ov{x}]}_i(H_j,K(p,m))=0$, 
		$\forall j\leq l,\,\forall i$. This is true for $l<<0$, as our complex is finite. Suppose, it is true for $j<l$. We have the spectral sequence:
		$$
		E_2^{p,q}=\op{Tor}^{P(m)[\ov{x}]}_p(H_q,K(p,m))
		$$
		converging to the cohomology of $(F,\partial)\otimes_{P(m)[\ov{x}]}K(p,m)$ which is zero, since $\ov{U}=0$. Here 
		$d_r:E_r^{p,q}\row E_r^{p-r,q+r-1}$. By our condition, there are
		no differential to, or from the $E^{0,l}$ term. Thus,
		$H_l\otimes_{P(m)[\ov{x}]}K(p,m)=E_2^{0,l}=E_{\infty}^{0,l}=0$.
		By Lemma \ref{Pm-Km-l2}, $\op{Tor}^{P(m)[\ov{x}]}_i(H_l,K(p,m))=0$,
		for any $i$. The induction step and the statement are proven.
		
		In particular, $H_0\otimes_{P(m)[\ov{x}]}K(p,m)=0$. Since 
		$H_0=\End(U)$, by Lemma \ref{Pm-Km-l1}, $Ann(U)\subset P(m)[\ov{x}]$ is non-zero.
		\Qed
	\end{proof}

	\begin{proposition}
		\label{vm-nilp-Km-triv}
		Let $k$ be flexible and $U\in\widehat{\SH}_{(p,m)}^c(k)$. Then the 
		following conditions are equivalent:
		\begin{itemize}
			\item[$(1)$] $Ann(\widehat{M}^{MGL}(U))\subset\laz$ contains some power of $v_m$ (where $v_{\infty}=1$).
			\item[$(2)$] $t_{K(p,m)}(\widehat{M}^{MGL}(U))=0\in K^b(Chow^{K(p,m)}_{Num}(k))$.
		\end{itemize}
	\end{proposition}
	
	\begin{proof}
		Let $m<\infty$, $M=\widehat{M}^{MGL}(U)$ and $t(M)$ be any choice of the weight complex for the weight filtration from Proposition 
		\ref{weight-str-M}. By \cite[Theorem 3.3.1.II]{Bon}, the radicals of the annihilators of $M$ and $t(M)$ in $\laz$ coincide. Thus,
		$(1)\Rightarrow$ $Ann(t(M))\subset\laz$ contains some power of $v_m$. Then the projection of $t(M)$ to 
		$K^b(Chow^{K(p,m)}_{Num}(k))$
		is zero, since $v_m$ is inverted there.
		
		Conversely, if the projection of $t(M)$ to $K^b(Chow^{K(p,m)}_{Num}(k))$ is zero, then by Proposition \ref{Ann-b}, the annihilator
		of $t(M)$ in $P(m)[\ov{x}]$ is non-zero, which means that
		$\sqrt{Ann(t(M))}\subset\laz$ contains $I(p,m)$ strictly. Since,
		by Proposition \ref{sq-J-U-p-r},
		$\sqrt{Ann(t(M))}=\sqrt{Ann(M)}$ is a prime invariant ideal $I(p,r)$, we have that $r>m$ and so, this radical contains $v_m$. 
		Hence, $Ann(M)$ contains some power of $v_m$. 
		
		For $m=\infty$, it follows from the conservativity of the functor $t_{K(p,\infty)}$.
		\Qed
	\end{proof}
	
	Combining this result with the properties of the {\it Morava weight cohomology} functor, we get:
	
	\begin{proposition}
		\label{SHhat-ideal}
		Let $k$ be flexible. Objects $U\in\widehat{\SH^c}_{\!\!(p,m)}(k/k)$ for which
		$\widehat{M}^{MGL}(U)$ is annihilated by some power of $v_m$ (where $v_{\infty}=1$)
		form a prime $\otimes\op{-}\triangle$-ed ideal. 
	\end{proposition}
	
	\begin{proof} 
		By Proposition \ref{vm-nilp-Km-triv}, this ideal consists of objects $U$ such that the Morava weight complex $t_{K(p,m)}$ of
		$\widehat{M}^{MGL}(U)$ is zero. 
		By Proposition \ref{Kpm-Num-zd}, the zero ideal $(0)$ of $K^b(Chow^{K(p,m)}_{Num}(k))$ is prime. Since the functors $\widehat{M}^{MGL}$ and $t_{K(p,m)}$ are tensor-triangulated, the
		preimage of this ideal is also prime.\\
		\phantom{a}\hspace{5mm}\Qed
	\end{proof}
	
	Now we can prove our Main Theorem.
	Consider {\it isotropic realisations}
	$$
	\psi_{(p,m),E}:\SH^c(k)\row \SH^c_{(p,m)}(\wt{E}/\wt{E}),
	$$
	where $p$ is a prime number, $1\leq m\leq\infty$, $E/k$ runs over $\stackrel{(p,m)}{\sim}$-equivalence classes of field extensions, and $\wt{E}$ is its {\it flexible closure}.
	
	\begin{theorem}
		\label{Main}
		Let ${\frak a}_{(p,m),E}:=\op{ker}(\psi_{(p,m),E})$. Then
		\begin{itemize}
			\item[$(1)$] ${\frak a}_{(p,m),E}$ are points of the Balmer spectrum $\Spc(\SH^c(k))$ of (the compact part) of Morel-Voevodsky category. 
			\item[$(2)$] ${\frak a}_{(p,m),E}={\frak a}_{(q,n),F}$ if and only if $p=q$, $m=n$ and $E/k\stackrel{(p,m)}{\sim}F/k$.
			\item[$(3)$] ${\frak a}_{(p,\infty),E}$ is the image of the point ${\frak a}_{p,E}$ of \cite[Theorem 5.13]{INCHKm} under
			the natural map of spectra $\Spc(\DM^c(k))\row\Spc(\SH^c(k))$
			induced by the motivic functor $M:\SH^c(k)\row \DM^c(k)$.
		\end{itemize}
	\end{theorem}
	
	\begin{proof}
		We will prove only parts $(1)$ and $(3)$ here and will leave the proof of $(2)$ to the next section.
		
		$(1)$: The isotropic realisation $\psi_{(p,m),E}$ is the
		composition
                $$
                \SH^c(k)\row \widehat{\SH}^c_{(p,m)}(\wt{E}/\wt{E})\row \SH^c_{(p,m)}(\wt{E}/\wt{E}).
                $$ 
		The preimage of the zero ideal $(0)$ of $\SH^c_{(p,m)}(\wt{E}/\wt{E})$ under the second map, by definition, consists of those
		objects $U\in \widehat{\SH}^c_{(p,m)}(\wt{E}/\wt{E})$ whose
		$\widehat{M}^{MGL}$-motive is annihilated by some power of $v_m$,
		for $m<\infty$, respectively, is zero, for $m=\infty$.
		By Proposition \ref{SHhat-ideal}, this ideal is prime. As 
		${\frak a}_{(p,m),E}$ is the preimage of it under the first map,
		it is also prime and defines a point of the Balmer spectrum.
		
		$(3)$ The ideal ${\frak a}_{(p,\infty),E}$ consists of $U\in \SH^c(k)$ with $t_{K(p,\infty)}(\widehat{M}^{MGL}(U_{\wt{E}}))=0$. But the latter is nothing else than the weight complex
		$t(\widehat{M}(U_{\wt{E}}))$ of the ($p$-isotropic) motive of $U$, where
		$\widehat{M}:\SH_{(p,\infty)}(\wt{E}/\wt{E})\row \DM(\wt{E}/\wt{E};\ff_p)$ is the (isotropic) motivic functor. Since the weight complex is conservative (as any object is an extension of its
		own shifted weight cohomology, see also \cite[Theorem 3.3.1.V]{Bon}), this happens if and only if $M(U)$ belongs to ${\frak a}_{p,E}$. Thus, ${\frak a}_{(p,\infty),E}$
		is the image of ${\frak a}_{p,E}$ under the map induced by $M$.
		\Qed
	\end{proof}
	
	\begin{example}
		\label{exa-norm-var}
		Let $\alpha=\{a_1,\ldots,a_n\}\neq 0\in K^M_n(k)/p$, and
		$Q_{\alpha}$ be the norm-variety of Rost \cite{RoNVAC}. Then the invariant $I_{BP}(Q_{\alpha})=\op{Im}(BP_*(Q_{\alpha})\stackrel{\pi_*}{\row}BP)$ is $I(n)=(v_0,\ldots,v_{n-1})$.
		Hence $Q_{\alpha}$ is $K(p,m)$-isotropic, for $m<n$, and 
		$K(p,m)$-anisotropic, for $m\geq n$, and remains the same over $\wt{k}$ - see Proposition \ref{anis-Pm-Km-IX}.
		
		In particular, $\Sigma^{\infty}_{\pp^1}Q_{\alpha}\in{\frak a}_{(p,m),k}$, for $m\geq n$. On the other hand, for any $m<n$, 
		there is $l$ prime to $p$, so that the map
		${\mathbbm{1}}^{MGL}(d)[2d]\stackrel{l\cdot v_m}{\lrow}
		{\mathbbm{1}}^{MGL}$, for $d=p^m-1$, factors through $M^{MGL}(Q_{\alpha})$ (by the definition of $I(Q_{\alpha})$). Since ${\mathbbm{1}}$ doesn't
		vanish in $K^b(Chow^{K(p,m)}_{Num}(\wt{k}))$ and $l\cdot v_m$ is
		invertible there, we get that $\Sigma^{\infty}_{\pp^1}Q_{\alpha}\not\in{\frak a}_{(p,m),k}$, for $m<n$.
		\Red
	\end{example}
	
	\section{Local stable homotopy categories}
	\label{section-lsh}
	
	To prove part $(2)$ of the Main Theorem \ref{Main} we will need
	to introduce, in analogy with \cite[Definition 2.3]{Iso}, the local versions $\SH_{(p,m)}(E/k)$ of $\SH(k)$ corresponding to all possible finitely-generated extensions $E/k$ of the base field,
	where isotropic categories correspond to trivial extensions.
	
	Let $E/k$ be a finitely-generated extension and $P/k$ be a smooth
	projective variety with $k(P)=E$. Let ${\mathbf{Q}}$ be the disjoint union of all smooth projective connected varieties $Q$, such
	that $k(Q)/k\stackrel{(p,m)}{\gneqq}E/k$. 
	Let $\check{C}({\mathbf{Q}})$ be the respective \v{C}ech simplicial scheme, and ${\frak X}_{{\mathbf{Q}}}$ be the
	$\Sigma^{\infty}_{{\pp^1}}$-spectrum of it. This is a $\wedge$-projector. Let $\wt{{\frak X}}_{{\mathbf{Q}}}$ be the complementary
	projector $\op{Cone}({\frak X}_{{\mathbf{Q}}}\row{\mathbbm{1}})$.
	Similarly, let ${\frak X}_P$ be the  
	$\Sigma^{\infty}_{{\pp^1}}$-spectrum of the \v{C}ech simplicial
	scheme of $P$ (note, that it depends only on the extension $E/k$ - see \cite[2.3.11]{SSWc}, so can be denoted as ${\frak X}_{E/k}$). 
	Then $\Upsilon_{(p,m),E/k}:=\wt{{\frak X}}_{{\mathbf Q}}\wedge
	{\frak X}_P$ is a $\wedge$-projector in $\SH(k)$.
	
	Define the category $\widehat{\SH}_{(p,m)}(E/k)$ as
	$\Upsilon_{(p,m),E/k}\wedge \SH(k)$. This is naturally a full subcategory of $\SH(k)$, which is a retract 
	of it via the above projector. Note though, that the respective
	functors are not adjoint to each other, in general.
	
	We have an adjoint pair:
	\begin{equation}
		\label{one-E-k}
		\xymatrix{
			M^{MGL}:\SH(k) \ar@<1ex>[rr]& & \MGL\op{-}mod:B\ar@<1ex>[ll]
		}.
	\end{equation}
	Denote $\widehat{\MGL}_{(p,m),E/k}\op{-}mod:=\Upsilon^{MGL}_{(p,m),E/k}\otimes\MGL\op{-}mod$, where
	$\Upsilon^{MGL}_{(p,m),E/k}:=M^{MGL}(\Upsilon_{(p,m),E/k})$, then we get an adjoint pair
	\begin{equation}
		\label{two-E-k}
		\xymatrix{
			\widehat{M}^{MGL}:\widehat{\SH}_{(p,m)}(E/k) \ar@<1ex>[rr]& & \widehat{\MGL}_{(p,m),E/k}\op{-}mod:\widehat{B}\ar@<1ex>[ll]
		}.
	\end{equation}
	
	Let $\ca$ be the localising subcategory of 
	$\widehat{\SH}_{(p,m)}(E/k)$ generated by objects
	$\Upsilon_{(p,m),E/k}\wedge U$, where $U\in\SH^c(k)$ and the
	$\widehat{M}^{MGL}$-motive of it belongs to the tensor localising ideal 
	of $\widehat{\MGL}_{(p,m),E/k}\op{-}mod$ generated by
	$\widehat{{\mathbbm{1}}}^{MGL}/v_m$, where $v_{\infty}=1$.
	
	\begin{definition}
		\label{isotr-cat-K-p-m-E-k}
		The isotropic stable homotopy category $\SH_{(p,m)}(E/k)$ is the Verdier localisation of the category $\widehat{\SH}_{(p,m)}(E/k)$ with respect to $\ca$. It comes equipped with the natural realisation functor
		$$
		\phi_{(p,m),E}:\SH(k)\row \SH_{(p,m)}(E/k).
		$$
	\end{definition}
	
	In the case of a trivial extension $k/k$, we get the isotropic stable homotopy category $\SH_{(p,m)}(k/k)$.
	
	There is functoriality for the denominator of the extension $E/k$. Namely, if $L\row F\row E$ is a tower of finitely-generated field extensions and $P/L$, $Q/L$ are smooth projective varieties,
	such that $L(P)=E$ and $L(Q)/L\stackrel{(p,m)}{\gneqq}E/L$,
	then $Q$ stays $K(p,m)$-anisotropic over $L(P)$ and so, over $F(P_F)$. Thus, $F(Q_F)/F\stackrel{(p,m)}{\gneqq}F(P_F)/F
	\stackrel{(p,m)}{\sim}E/F$. Moreover, since there are rational maps in both directions between $P_F/F$ and a smooth model of $E/F$,
	we have ${\frak X}_{F(P_F)/F}={\frak X}_{E/F}$. Hence, we get a 
	natural functor
	$$
	\SH_{(p,m)}(E/L)\row \SH_{(p,m)}(E/F).
	$$
	In particular, we obtain the local to isotropic functor
	$\SH_{(p,m)}(E/k)\row \SH_{(p,m)}(E/E)$. In order to evaluate
	the conservativity of it, we will need the following fact (cf. \cite[Proposition 2.7]{Iso}).
	
	We will use the notations: $\hi_R:=M^{MGL}({\frak X}_R)$.
	
	\begin{proposition}
		\label{ker-loc-var-fin-type}
		Suppose, $U\in \SH^c(k)$ vanishes in $\SH_{(p,m)}(E/k)$, then
		there exists a smooth variety $Q/k$ of finite type, such that
		$k(Q)/k\stackrel{(p,m)}{\gneqq}E/k$ and
		$\wt{\hi}_Q\otimes\hi_P\otimes M^{MGL}(U)$ belongs to the tensor localising ideal generated by  
		${\mathbbm{1}}^{MGL}/v_m$, where $v_{\infty}=1$.
	\end{proposition}
	
	\begin{proof}
		We know that 
		$\wt{\hi}_{{\mathbf{Q}}}\otimes\hi_P\otimes M^{MGL}(U)$ belongs to the tensor localising ideal ${\mathcal I}$ generated by 
		${\mathbbm{1}}^{MGL}/v_m$. Then 
		$\wt{\hi}_{{\mathbf{Q}}}\otimes M^{MGL}(P)\otimes M^{MGL}(U)$
		also belongs to the same ideal (as
		$\hi_P\otimes M^{MGL}(P)=M^{MGL}(P)$, because the projection $P\times P\row P$ has a splitting - the diagonal, see \cite[Lemma 1.15]{MV}). Since $M:=M^{MGL}(P)\otimes M^{MGL}(U)$ and the generator ${\mathbbm{1}}^{MGL}/v_m$ of
		${\mathcal I}$ are compact, the map $M\row
		\wt{\hi}_{{\mathbf{Q}}}\otimes M$
		factors throughs some compact object $N$ of ${\mathcal I}$, and so,
		is the composition 
		$$
		M\row N\row\wt{\hi}_Q\otimes M\row\wt{\hi}_{{\mathbf{Q}}}\otimes M,
		$$
		for some variety $Q$ of finite type as above. Moreover, (by passing to a larger $Q$, if necessary) we may 
		assume that the map $M\row\wt{\hi}_Q\otimes M$ here is
		$({\mathbbm{1}}^{MGL}\row\wt{\hi}_Q)\otimes id_M$.
		Tensoring with $\wt{\hi}_Q$, we see that the identity map of
		$\wt{\hi}_Q\otimes M$ factors through $\wt{\hi}_Q\otimes N\in{\mathcal I}$. Hence, $\wt{\hi}_Q\otimes M$ is a direct summand
		of $\wt{\hi}_Q\otimes N$ and belongs to ${\mathcal I}$.
		
		Since $\wt{\hi}_Q\otimes M^{MGL}(P)\otimes M^{MGL}(U)$ belongs
		to ${\mathcal I}$ and ${\hi}_P$ belongs to the tensor localising ideal generated by $M^{MGL}(P)$, we get that
		$\wt{\hi}_Q\otimes{\hi}_P\otimes M^{MGL}(U)$ belongs
		to ${\mathcal I}$.
		\Qed
	\end{proof}
	
	In analogy with \cite[Proposition 2.9]{Iso} we have:
	
	\begin{proposition}
		\label{D-A9}
		Let $E/k$ be a finitely-generated extension of a flexible field. Then the functor
		$$
		\SH_{(p,m)}(E/k)\lrow \SH_{(p,m)}(E/E)
		$$
		is conservative on the image of $\phi_{(p,m),E}$ on $\SH^c(k)$.
	\end{proposition}
	
	\begin{proof}
		We follow the proof of \cite[Proposition 2.9]{Iso}.
		Any finitely generated extension is a tower of a purely transcendental extension followed by a finite one.
		We will first treat the purely transcendental part.
		
		\begin{lem}
			\label{D-A7}
			Let $E/L/k$ be a tower of finitely generated extensions where $L/k$ is a purely transcendental extension of a flexible field.
			Then the functor
			$$
			\SH_{(p,m)}(E/k)\lrow \SH_{(p,m)}(E/L)
			$$
			is conservative on the image of $\phi_{(p,m),E}$ on $\SH^c(k)$.
		\end{lem}
		
		\begin{proof}
			Let $L=k(\aaa^n)$, and $E=k(\ov{R})$ for some smooth variety $\ov{R}/k$.
			Let $\ov{U}\in Ob({\frak X}_{\ov{R}}\wedge \SH^c(k))$ be an object vanishing in $\SH_{(p,m)}(E/L)$. Then,
			according to the Proposition \ref{ker-loc-var-fin-type}, there exists a
			variety $\ov{Q}/L$ of finite type such that $L(\ov{Q})/L\stackrel{(p,m)}{\gneqq} L(\ov{R})/L$ and $\wt{\hi}_{\ov{Q}}\otimes M^{MGL}(\ov{U_L})$ belongs to the tensor localising ideal generated by ${\mathbbm{1}}^{MGL}/v_m$.
			
			The condition $L(\ov{Q})/L\stackrel{(p,m)}{\gneqq} L(\ov{R})/L$ means that we have a $K(p,m)$-$L$-correspondence
			$\ov{\alpha}:\ov{Q}\rightsquigarrow \ov{R}_L$ of degree one, and there is no such correspondence in the opposite direction.
			Since $k=k_0(\pp^{\infty})$ is flexible,
			varieties $\ov{R}$ and $\ov{Q}$ are, actually, defined over $F$ and $M=F(\aaa^n)$, respectively, where extensions $k/F/k_0$ are
			purely transcendental and $F/k_0$ is moreover finitely generated. By the same reason, we can assume that
			the object $\ov{U}$ is defined over $F$, while
			the correspondence $\ov{\alpha}$ is defined over $M$.
			So, there exist varieties $R/F$, $Q/M$, an object $U$ of ${\frak X}_{R}\wedge \SH^c(F)$ and a degree one $K(p,m)$-$M$-correspondence
			$\alpha:Q\rightsquigarrow R_M$ such that $R|_k=\ov{R}$, $Q|_L=\ov{Q}$, $U|_k=\ov{U}$ and $\alpha|_L=\ov{\alpha}$.
			Note, that we still have: $M(Q)/M\stackrel{(p,m)}{\gneqq} M(R)/M$ (since $\alpha$ is defined over $M$ and by functoriality). 
			
			Let $V=M^{MGL}(R)\otimes M^{MGL}(U)$. We know that the object $W=\wt{\hi}_Q\otimes M^{MGL}(U_M)$, when restricted to $L$, belongs to the $\otimes$-localising
			ideal of $\MGL(L)\op{-}mod$ generated by ${\mathbbm{1}}^{MGL}/v_m$. Then the same is
			true about $Y=M^{MGL}(R_M)\otimes W=\wt{\hi}_Q\otimes V_M$. But $V$ is compact. As in the 
			proof of Proposition \ref{ker-loc-var-fin-type}, we get that $Y_L$ is a direct summand of $\wt{\hi}_Q\otimes N$, for some compact representative $N$ of the ideal. Hence, $Y$
			belongs to the described ideal already over some finitely generated purely transcendental extension $M'=F'(\aaa^n)$ of $M=F(\aaa^n)$, where $N$ and morphisms involved are defined (here the extensions $k/F'/F$ are purely transcendental). Then the same happens to $W$ (recall, that $W$ is divisible by $\hi_{R_M}$, and so, belongs to the $\otimes$-localising ideal generated by $Y$).

			The extension
			$M'/F'$ can be embedded into $k/F'$ making $k/M'$ purely transcendental. Let $Q'$ be a variety over $k$ obtained from
			$Q_{M'}$ using this embedding. Then $k(Q')/k\stackrel{(p,m)}{\gneqq} k(\ov{R})/k$ (as $k/M$ is purely transcendental).
			
			Since $\wt{\hi}_{Q'}\otimes M^{MGL}(\ov{U})$ belongs to the $\otimes$-localising
			ideal of $\MGL(k)\op{-}mod$ generated by ${\mathbbm{1}}^{MGL}/v_m$, we obtain that 
			$\ov{U}=0$ in $\SH_{(p,m)}(E/k)$.
			\Qed
		\end{proof}
		
		Using Lemma \ref{D-A7} our problem is reduced to
		the case of a finite extension. In this situation, the statement is true for an arbitrary field.
		
		\begin{lem}
			\label{D-A8}
			Let $E/L$ be a finite extension of fields. Then the functor
			$$
			\SH_{(p,m)}(E/L)\lrow \SH_{(p,m)}(E/E)
			$$
			is conservative.
		\end{lem}
		
		\begin{proof}
			Let $E=L(P)$ for some smooth connected 0-dimensional variety $P$.
			Let $U\in Ob(\SH^c(L))$ be an object
			vanishing in $\SH_{(p,m)}(E/E)$. Then, for the  disjoint union $\ov{{\mathbf{Q}}}$ of all anisotropic varieties over $E$,
			we have: $\wt{\hi}_{\ov{{\mathbf{Q}}}}\otimes M^{MGL}(U_E)$ belongs to the tensor localising ideal generated by ${\mathbbm{1}}^{MGL}/v_m$. Consider a smooth $L$-variety $\widehat{{\mathbf{Q}}}$
			given by the composition
			$\ov{{\mathbf{Q}}}\row\op{Spec}(E)\row\op{Spec}(L)$. We have a natural map $\ov{{\mathbf{Q}}}\row\widehat{{\mathbf{Q}}}_E$, and so, $\wt{\frak{X}}_{\ov{{\mathbf{Q}}}}\wedge\wt{\frak{X}}_{\widehat{{\mathbf{Q}}}_E}=\wt{\frak{X}}_{\widehat{{\mathbf{Q}}}_E}$. Clearly, $L(\widehat{{\mathbf{Q}}})/L\stackrel{(p,m)}{\geq}L(P)/L$.
			Suppose, these are equivalent. Then in the commutative
			diagram
			$$
			\xymatrix{
				\ov{{\mathbf{Q}}} \ar[r] & \op{Spec}(E) \ar[r] & \op{Spec}(L) \\
				\widehat{{\mathbf{Q}}}_E \ar[u] \ar[r] & \op{Spec}(E\otimes_L E) \ar[u] \ar[r] & \op{Spec}(E) \ar[u]
			}
			$$
			the $K(p,m)$-push-forward along the total lower horizontal map is surjective. Here $\op{Spec}(E\otimes_L E)$ is a disjoint union of 
			$\op{Spec}(F)$ for some fields $F$. But since $[E:L]$ is finite, the degrees of both maps $\op{Spec}(E)\llow\op{Spec}(F)\lrow\op{Spec}(E)$
			coincide. In particular, the respective $K(p,m)$-push-forwards coincide as
			well. Hence, the $K(p,m)$-push-forward along the map
			$\ov{{\mathbf{Q}}}\row\op{Spec}(E)$ is surjective - a contradiction, as
			$\ov{{\mathbf{Q}}}$ is $K(p,m)$-anisotropic.
			Hence, $\widehat{{\mathbf{Q}}}_E$ is $K(p,m)$-anisotropic and so, there is a map $\widehat{{\mathbf{Q}}}\row {\mathbf{Q}}$ (over $L$), 
			where ${\mathbf{Q}}$ is a disjoint union of all $L$-varieties $Q$
			with $L(Q)/L\stackrel{(p,m)}{\gneqq}L(P)/L$. In particular,
			$\wt{{\frak{X}}}_{\widehat{{\mathbf{Q}}}}\wedge\wt{{\frak X}}_{{\mathbf{Q}}}=
			\wt{{\frak X}}_{{\mathbf{Q}}}$.
			
		    Since $\wt{{\frak X}}_{\ov{{\mathbf{Q}}}}\wedge\wt{{\frak X}}_{{\mathbf{Q}}_E}=\wt{{\frak X}}_{{\mathbf{Q}}_E}$, 
			we know that $(\wt{\hi}_{{\mathbf{Q}}}\otimes M^{MGL}(U))_E$ belongs to the tensor localising ideal generated by 
			${\mathbbm{1}}^{MGL}/v_m$. Then the same is true about
			$M^{MGL}(P)\otimes\wt{\hi}_{{\mathbf{Q}}}\otimes M^{MGL}(U)$,
			and so, also for $\hi_P\otimes\wt{\hi}_{{\mathbf{Q}}}\otimes M^{MGL}(U)$.
			Thus, $U=0$ in $\SH_{(p,m)}(E/L)$.
			\Qed
		\end{proof}
		
		This finishes the proof of the Proposition \ref{D-A9}.
		\Qed
	\end{proof}
	
	Combining the restriction $k\row\wt{k}$ with $\phi_{(p,m),\wt{E}}$, we get the
	{\it local realisation}:
	$$
	\ffi_{(p,m),E}:\SH(k)\row \SH_{(p,m)}(\wt{E}/\wt{k}).
	$$
	From Proposition \ref{D-A9} we immediately get:
	
	\begin{corollary}
		\label{fi-psi-ker}
		$$
		{\frak{a}}_{(p,m),E}=\op{ker}(\ffi^c_{(p,m),E}).
		$$
	\end{corollary}
	
	\begin{proof}
		Indeed, since $\wt{k}$ is flexible, by Proposition \ref{D-A9},
		$$
		\op{ker}(\SH^c(\wt{k})\row \SH_{(p,m)}(\wt{E}/\wt{k}))=\op{ker}(\SH^c(\wt{k})\row \SH_{(p,m)}(\wt{E}/\wt{E})).
		$$
		Then their pre-images under $\SH^c(k)\row \SH^c(\wt{k})$ coincide as well.
		\Qed
	\end{proof}
	
	Now we can compare these ideals for $K(p,m)$-equivalent extensions (not necessarily finitely-generated).
	
	\begin{proposition}
		\label{ideals-Kpm-eq-ext}
		Let $E/k$ and $F/k$ be arbitrary field extensions. Then
		$$
		E/k\stackrel{(p,m)}{\sim}F/k\hspace{5mm}\text{if and only if}
		\hspace{5mm}{\frak{a}}_{(p,m),E}={\frak{a}}_{(p,m),F}.
		$$
	\end{proposition}
	
	\begin{proof}
		Let $E/k\stackrel{(p,m)}{\sim}F/k$ be two $K(p,m)$-equivalent field extensions. Then $\wt{E}/\wt{k}\stackrel{(p,m)}{\sim}\wt{F}/\wt{k}$ as well. 
		Let $E=\op{colim}_{\alpha}E_{\alpha}$ and 
		$F=\op{colim}_{\beta}F_{\beta}$, where $E_{\alpha}=k(R_{\alpha})$ and
		$F_{\beta}=k(P_{\beta})$ are finitely-generated extensions with
		smooth models $R_{\alpha}$ and $P_{\beta}$.
		
		Let $U\in {\frak{a}}_{(p,m),E}$. Then, since $M^{MGL}(U_{\wt{E}})$ is compact, by Proposition \ref{altern-def-isotr-sh},
		it is a direct summand of some extension of $MGL$-motives of a finite number of $K(p,m)$-anisotropic varieties $S$
		over $\wt{E}$ as well as finitely many
		objects of the type $M^{MGL}(Y)\otimes{\mathbbm{1}}^{MGL}/v_m$,
		for some $\wt{E}$-varieties $Y$. All the varieties and 
		morphisms involved are defined over some finitely-generated
		extension $\wt{E}_{\alpha}=\wt{k}(R_{\alpha})$. 
		Since $E/k\stackrel{(p,m)}{\sim}F/k$,
		the respective varieties $S$ will stay $K(p,m)$-anisotropic
		over $\wt{F}(R_{\alpha})$. Thus, $U_{\wt{F}(R_{\alpha})}\in
		{\frak a}_{(p,m),\wt{F}(R_{\alpha})/\wt{F}(R_{\alpha})}$.
		By Proposition \ref{D-A9}, $U_{\wt{F}}$ vanishes when restricted
		to $\SH_{(p,m)}(\wt{F}(R_{\alpha})/\wt{F})$. But since 
		$F/k\stackrel{(p,m)}{\geq}E/k$, the variety
		$R_{\alpha}|_{\wt{F}}$ is $K(p,m)$-isotropic and so,
		$\wt{F}(R_{\alpha})/\wt{F}$ and $\wt{F}/\wt{F}$ are
		$K(p,m)$-equivalent finitely-generated field extensions.
		So, in order to show that $U\in{\frak{a}}_{(p,m),F}$,
		in other words, that $U_{\wt{F}}$ vanishes when restricted 
		to $\SH_{(p,m)}(\wt{F}/\wt{F})$,
		it is sufficient to prove the finitely-generated case 
		of our statement. Below we assume that $\wt{E}=\wt{k}(R)$ and $\wt{F}=\wt{k}(P)$
		are finitely-generated and $K(p,m)$-equivalent (over $\wt{k}$).
		Then, the varieties ${\mathbf Q}$ are the same, $\Upsilon_{(p,m),\wt{E}/\wt{k}}={\frak{X}}_R\wedge\wt{{\frak{X}}}_{{\mathbf Q}}$ and $\Upsilon_{(p,m),\wt{F}/\wt{k}}={\frak{X}}_P\wedge\wt{{\frak{X}}}_{{\mathbf Q}}$.
		
		We know that $I(P_{\wt{E}})=\op{Im}(\Omega_*(P_{\wt{E}})\stackrel{\pi_*}{\lrow}\Omega_*(\op{Spec}(\wt{E})))\subset\laz$ is a finitely generated ideal invariant under Landweber-Novikov operations. So, by the results of Landweber, the radical $\sqrt{I(P_{\wt{E}})}$ of it is an intersection of finitely many ideals $I(q,n)$ of Landweber. Since 
		$\wt{E}/\wt{k}\stackrel{(p,m)}{\geq}\wt{F}/\wt{k}$, 
		the variety $P_{\wt{E}}$ is $K(p,m)$-isotropic, then
		by Proposition \ref{anis-Pm-Km-IX},
		this radical localised at $p$ contains
		$v_m$ and the $p$-component of it is $I(p,n)$, for some $n>m$, or the whole ring $\laz_{(p)}$. 
		This shows that there is an $\Omega^*$-cycle on $P\times R$
		whose projection to $R$ is $l\cdot v_m^r\cdot 1_R$, for some $r$
		and some $l$ relatively prime to $p$.
		In other words, we have a commutative diagram
		$$
		\xymatrix{
			M^{MGL}(R)(*)[2*] \ar[d] \ar[r] & T(*)[2*] \ar[d]^{l\cdot v_m^r}\\
			M^{MGL}(P) \ar[r] & T
		}
		$$
		where $T=M^{MGL}(\op{Spec}(\wt{k}))$ and horizontal maps are induced by projections to the point. Multiplied by $M^{MGL}(R)$ the upper map acquires a 
		splitting (given by the diagonal map) and we see that the
		map $M^{MGL}(R)(*)[2*]\stackrel{l\cdot v_m^r}{\lrow}M^{MGL}(R)$
		factors through $M^{MGL}(P\times R)$. 
		Note that as $\wt{k}$ is flexible, there are non-trivial extensions $L/\wt{k}$ of degree $p$ which stay anisotropic over $\wt{E}$ and $\wt{F}$.
		Thus, the identity map of $\wt{\hi}_{{\mathbf Q}}$ is killed by $p$, and we may get rid of the prime to $p$ number $l$, if we tensor with 
		$\wt{\hi}_{{\mathbf Q}}$. That is, the map 
		$$
                M^{MGL}(R)\otimes\wt{\hi}_{{\mathbf Q}}(*)[2*]\stackrel{v_m^r}{\lrow}M^{MGL}(R)\otimes\wt{\hi}_{{\mathbf Q}}
               $$ 
                factors 
		through $M^{MGL}(P\times R)\otimes\wt{\hi}_{{\mathbf Q}}$.
		Since $\wt{F}/\wt{k}\stackrel{(p,m)}{\geq}\wt{E}/\wt{k}$ as well,
		the map 
		$$
                M^{MGL}(P)\otimes\wt{\hi}_{{\mathbf Q}}(*')[2*']\stackrel{v_m^s}{\lrow}M^{MGL}(P)\otimes\wt{\hi}_{{\mathbf Q}}
		$$ 
 		factors 
		through $M^{MGL}(P\times R)\otimes\wt{\hi}_{{\mathbf Q}}$, for some $s$. Observe, that the cones of these two maps
		belong to the tensor localising ideal generated by ${\mathbbm{1}}^{MGL}/v_m$. It follows that, for a given
		object $U\in \SH^c(\wt{k})$, if
		$M^{MGL}(U)\otimes\hi_R\otimes\wt{\hi}_{{\mathbf{Q}}}$ belongs
		to this ideal, then so does
		$M^{MGL}(U)\otimes M^{MGL}(R)\otimes\wt{\hi}_{{\mathbf{Q}}}$
		and hence, 
		$M^{MGL}(U)\otimes M^{MGL}(P)\otimes\wt{\hi}_{{\mathbf{Q}}}$
		and 
		$M^{MGL}(U)\otimes \hi_P\otimes\wt{\hi}_{{\mathbf{Q}}}$. 
		And vice-versa. 
		Thus, $\op{ker}(\ffi^c_{(p,m),E})=\op{ker}(\ffi^c_{(p,m),F})$.

		Conversely, if $E/k\stackrel{(p,m)}{\not\sim}F/k$, then
		$E=k(R)$, $F=k(P)$, where (at least) one of the varieties
		$R_F$, $P_E$ is $K(p,m)$-anisotropic. From symmetry, may assume 
		the former. Then $\Sigma^{\infty}_{\pp^1}R_{+}\in {\frak{a}}_{(p,m),F}$ (as $R$ remains anisotropic over $\wt{F}$).
		On the other hand, a Tate-motive splits from $M^{MGL}(R)$ over $E$, and so, the image of this object in $\widehat{\MGL}_{(p,m)}\!-\!mod$ is not annihilated by any power of $v_m$. Hence, $\Sigma^{\infty}_{\pp^1}R_{+}\not\in {\frak{a}}_{(p,m),E}$ and the ideals are different.
		\Qed
	\end{proof}

	Now we are ready to prove part $(2)$ of the Main Theorem \ref{Main}.

	$(2)$: We have a natural homomorphism of spectra
	$\rho:\Spc(\SH^c(k))\row\op{Spec}(\op{End}_{\SH(k)}({\mathbbm{1}}))$ - see \cite[Corollary 5.6]{BalSSS},
	where $\op{End}_{\SH(k)}({\mathbbm{1}})=\op{GW}(k)$. The
	image of ${\frak{a}}_{(p,m),E}$ is the ideal consisting of those
	virtual quadratic forms $\ffi$ that ${\mathbbm{1}}/\ffi:=\op{Cone}({\mathbbm{1}}\stackrel{\ffi}{\row}{\mathbbm{1}})$ is not contained
	in ${\frak{a}}_{(p,m),E}$. Suppose, $\ffi$ has dimension $d$ divisible by $p$. Consider the projection of ${\mathbbm{1}}/\ffi$
	via the composition
	$$
	\SH^c(k)\hspace{-1mm}\row\widehat{\SH}^c_{(p,m)}(\wt{E}/\wt{E})\hspace{-1mm}
	\stackrel{\widehat{M}^{MGL}}{\row}\hspace{-1.5mm}\widehat{\MGL}^c_{(p,m)}\op{-}mod(\wt{E})\hspace{-1mm}\stackrel{t_{K(p,m)}}		 {\lrow}\hspace{-1.5mm}K^b(Chow^{K(p,m)}_{Num}(\wt{E})).
	$$
	The category $\widehat{\MGL}^c_{(p,m)}(\wt{E})\op{-}mod$ is $p$-torsion,
	so $\ffi$ projects to zero there. Hence, the projection of
	${\mathbbm{1}}/\ffi$ to
	$K^b(Chow^{K(p,m)}_{Num}(\wt{E}))$ is a direct sum of two Tate-motives, and so, is non-zero. On the other hand, for $\ffi=\la 1\ra$
	(the unit of $\op{GW}(k)$), the object ${\mathbbm{1}}/\la1\ra$ is zero
	already in $SH^c(k)$. Hence, $\rho({\frak a}_{(p,m),E})={\frak q}_p$ is the image of $(p)$ under the natural map 
	$\op{Spec}(\zz)\row\op{Spec}(\op{GW}(k))$ induced by the rank map
	$rank:\op{GW}(k)\row\zz$ mapping $\ffi$ to $\ddim(\ffi)$. Hence, if ${\frak{a}}_{(p,m),E}={\frak{a}}_{(q,n),F}$, then $p=q$. 
	
	The invariant $m$ of ${\frak{a}}_{(p,m),E}$ can be recovered with the help of certain test spaces. Namely, by \cite[Theorem 3.1]{TS} (which is an adaptation to the algebro-geometric context of the construction of Mitchell \cite{Mit}), for any field $k$ of characteristic different from $p$, there exists an object $X_{p,n}$
	in $\SH^c(k)_{(p)}$, whose $p$-level is exactly $n$, that is, the $r$-th 
	Morava K-theory vanishes on $X_{p,n}$, for $r<n$, and is 
	non-zero, for $r\geq n$. Moreover, for any $r<n$, 
	$M^{MGL}(X_{p,n})$ is killed by some power of $v_r$. Such an object can be constructed as a direct summand of the suspension spectrum of some smooth variety, namely, of the quotient
	$V\backslash GL(p^{n})$, where $V=(\zz/p)^{n}$ acts by permutations on the basis, for odd $p$ (respectively, of
	$V\backslash GL(2^{n+1})$, where $V=(\zz/2)^{n+1}$, for $p=2$). The important feature of $X_{p,n}$ is
	that the properties of it are unaffected by field extensions.
	The number $m$ is uniquely determined from the following result.
	
	\begin{proposition}
		\label{m-Xpn}
		$X_{p,n}\in {\frak{a}}_{(p,m),E}$, for any $n>m$, and
		$X_{p,n}\not\in {\frak{a}}_{(p,m),E}$, for any $n\leq m$.
	\end{proposition}
	
	\begin{proof}
		Since $M^{MGL}(X_{p,n})$ is killed by some power of $v_m$ by
		\cite[Theorem 3.2]{TS}, 
		$X_{p,n}\in {\frak{a}}_{(p,m),E}$, for any $n>m$. On the other
		hand, if $n\leq m$, then $\widehat{M}^{MGL}(X_{p,n})\in\widehat{\MGL}_{(p,m)}(\wt{E})\!-\!mod$ is not killed by any power of $v_m$, by \cite[Corollary 3.6]{TS}, and so,
		$X_{p,n}\not\in {\frak{a}}_{(p,m),E}$.
		\Qed
	\end{proof}
	
	It follows from Proposition \ref{m-Xpn} that if 
	${\frak{a}}_{(p,m),E}={\frak{a}}_{(p,n),F}$, then $n=m$.
	Combining it with the Proposition \ref{ideals-Kpm-eq-ext},
	we complete the proof of part $(2)$ of the Main Theorem.
	\Qed
	
	\begin{remark}
		\label{strange-iso-spec}
		Observe, that we don't have the classical specialisation relations among isotropic points. Indeed, while their topological counterparts ${\frak a}_{(p,m),Top}$ satisfy:
		${\frak a}_{(p,n),Top}\subset{\frak a}_{(p,m),Top}$, for $n>m$,
		and so, the former point is a specialisation of the latter,
		this doesn't hold for isotropic points ${\frak a}_{(p,m),E}$.
		The obstruction comes (for example) from the (suspension spectra) of norm-varieties - see Example \ref{exa-norm-var}. Indeed, if $n>m$ and $\alpha\in K^M_n(k)/p$ is a non-zero pure symbol, then the suspension spectrum $\Sigma^{\infty}_{\pp^1}(Q_{\alpha})$ of
		the norm-variety $Q_{\alpha}$ belongs to ${\frak a}_{(p,n),k}$,
		but not to ${\frak a}_{(p,m),k}$.
		
		The reason for this is that the isotropic ideals are generated
		not only by $v_m$-nilpotent objects, as in topology, but also
		by $K(p,m)$-anisotropic ones. These, in a sense, counterbalance each other - compare the properties of torsion
		spaces $X_{p,n}$ from Proposition \ref{m-Xpn} with those of
		$\Sigma^{\infty}_{\pp^1}(Q_{\alpha})$ from Example \ref{exa-norm-var}.
		\Red 
	\end{remark}
	
	\begin{remark}
		It is not enough to consider finitely-generated extensions only.
		The vast majority of $(p,m)$-equivalence classes will
		not have any finitely-generated representatives -
		see \cite[Example 5.14]{INCHKm}.
		\Red
	\end{remark}
	
	For $m\leq n$, the $K(p,m)$-isotropic varieties are 
	$K(p,n)$-isotropic (see Proposition \ref{anis-Pm-Km-IX}), hence
	the equivalence relation $\stackrel{(p,m)}{\sim}$ is {\it coarser} than $\stackrel{(p,n)}{\sim}$. Thus, the set of
	$K(p,n)$-isotropic points naturally surjects to the set of
	$K(p,m)$-isotropic ones. So, in the end, all these Morava-points
	${\frak a}_{(p,m),E}$ form a ``quotient'' of motivic isotropic
	points ${\frak a}_{p,E}$. 
	The following example shows that these quotients are still pretty
	large.
	
	\begin{example} {\rm (cf. \cite[Example 5.14]{INCHKm})}
		\label{real-case}
		For an algebraically closed field $k$, we have a single 
		$\stackrel{(p,m)}{\sim}$-equivalence class of extensions
		(as all varieties have rational points and so, are $K(p,m)$-isotropic), and hence, a single {\it isotropic point}, for
		every $p$ and $m$.
		
		For $k=\rr$, we have a unique 
		$\stackrel{(p,m)}{\sim}$-equivalence class of extensions,
		for any odd $p$. But for $p=2$ the situation is completely different. Observe, that a curve of genus one is $K(p,1)$-isotropic if and only if it is $K(p,\infty)$-isotropic. This
		follows from Proposition \ref{anis-Pm-Km-IX} and the fact that
		the ideal $I(X)\subset BP$ is generated by the images of the unit $1_X\in BP(X)$ (which is zero, in our case) and the classes of all closed points. Thus, the extensions $E_J$ constructed in 
		\cite[Example 5.14]{INCHKm} are not $K(2,m)$-equivalent,
		for any $1\leq m\leq\infty$. Hence, for any such $m$, we get
		$2^{2^{\aleph_0}}$ different $K(2,m)$-isotropic points
		${\frak a}_{(2,m),E}$.
	\end{example}
	
	\begin{remark}
	 The isotropic points are different from the topological points. This can be seen by their action on 
	 $\tau:{\mathbbm{1}}/p(-1)\row{\mathbbm{1}}/p$. While the
	 topological realisation inverts $\tau$, the isotropic one
	 maps it to zero. Thus, $\op{Cone}(\tau)\in{\frak a}_{(p,m),Top}$, while $\op{Cone}(\tau)\not\in{\frak a}_{(p,m),E}$, for any $p,m$ and $E$. The latter fact is obvious, since $A:=t_{K(p,m)}({\mathbbm{1}}/p)\in K^b(Chow^{K(p,m)}_{Num}(\wt{E}))$ is non-zero. More precisely, is a sum of two Tate-motives, which are exactly the {\it Morava weight cohomology} of it 
	 (Definition \ref{Morava-weight-coh}). Then no map
	 $f:A(-1)\row A$ may be invertible, since it shifts the
	 weight grading. By Proposition \ref{vm-nilp-Km-triv},
	 $\op{Cone}(\tau)\not\in{\frak a}_{(p,m),E}$.
	 \Red
	\end{remark}

	\bigskip
	
	\begin{itemize}
		\item[address:] Peng Du: {\small Institute of Mathematics, Henan Academy of Sciences, No. 228, Chongshili, Zhengdong New District, Zhengzhou, Henan 450046, P. R. China}
         \item[] Alexander Vishik: {\small School of Mathematical Sciences, University of Nottingham, University Park, Nottingham, NG7 2RD, UK}
		\item[email:] {\small\ttfamily pengdudp@gmail.com}
		\item[] {\small\ttfamily alexander.vishik@nottingham.ac.uk}
	\end{itemize}

\end{document}